\documentclass{amsart}
\usepackage{mathrsfs,amssymb,mathrsfs, multirow,setspace}
\usepackage[a4paper, total={7in, 9in}]{geometry}
\usepackage{enumerate}
\theoremstyle{plain}
\usepackage{graphicx}
\usepackage{mathtools}

\newtheorem{theorem}{Theorem}[section]
\newtheorem{lemma}[theorem]{Lemma}
\newtheorem{proposition}[theorem]{Proposition}

\theoremstyle{definition}

\theoremstyle{remark}

\input xy
\xyoption{all}

\begin{document}

	\title[Characterization of rings with genus two cozero-divisor graphs]{Characterization of rings with genus two cozero-divisor graphs}
 
	\author[Praveen Mathil, Barkha Baloda, Jitender Kumar]{Praveen Mathil, Barkha Baloda, Jitender Kumar*}
	\address{Department of Mathematics, Birla Institute of Technology and Science Pilani, Pilani 333031, India}
	\email{maithilpraveen@gmail.com, barkha0026@gmail.com, jitenderarora09@gmail.com}

\begin{abstract}
Let $R$ be a ring with unity. The cozero-divisor graph of a ring $R$ is an undirected simple graph whose vertices are the set of all non-zero and non-unit elements of $R$ and two distinct vertices $x$ and $y$ are adjacent if and only if $x \notin Ry$ and $y \notin Rx$. The reduced cozero-divisor graph of a ring $R$ is an undirected simple graph whose vertex set is the set of all nontrivial principal ideals of $R$ and two distinct vertices $(a)$ and $(b)$ are adjacent if and only if $(a) \not\subset (b)$ and $(b) \not\subset (a)$. In this paper, we characterize all classes of finite non-local commutative rings for which the cozero-divisor graph and reduced cozero-divisor graph are of genus two.

\end{abstract}

\keywords{Non-local ring, ideals, genus, nilpotent index of ideal.\\ *  Corresponding author}
\maketitle

\section{Historical background and main results}
The study of algebraic structures through the properties of their associated graphs has appeared as a fascinating research topic in the past three decades. Associating a graph to an algebraic structure leads to an interplay between algebraic properties of respective algebraic structure and graph theoretic properties of their associated graphs. Motivated by certain applications of graphs associated to algebraic structures (see \cite{cooperman1990applications, radha2021frequency}), various authors introduced some more graphs associated to algebraic objects. Numerous aspects of certain graphs  including zero-divisor graph, cozero-divisor graph, total graph, co-maximal graph, annihilating-ideal graph etc. (see \cite{afkhami2011cozero,afkhami2012generalized, akbari2009total, anderson2008zero, anderson2008total, anderson1999zero, behboodi2011annihilating, biswas2022subgraph, maimani2008comaximal}), associated to ring structures have been studied in the literature.  Topological graph theory is mainly concerned with the embedding of a graph into a surface without edge crossing. Its applications lie in electronic printing circuits where the purpose is to embed a circuit, that is, the graph on a circuit board (the surface), without two connections crossing each other, resulting in a short circuit. Determining the genus $g(\Gamma)$ of a graph $\Gamma$ is a fundamental but highly complex problem. It is indeed NP-complete. Numerous authors have studied the problem of finding the genus of various algebraic graphs. For example, the genus of zero divisor graphs of rings has been investigated in \cite{akbari2003zero, asir2020classification, belshoff2007planar,  wang2006zero, wickham2008classification, wickham2009rings}. The research work on the genus of various other graphs associated to rings can be found in \cite{asir2014genus, maimani2012rings, selvakumar2018commutative, chelvam2013genus} and references therein.
 
 Afkhami \emph{et al.} \cite{afkhami2011cozero} introduced the cozero-divisor graph of a commutative ring. The \emph{cozero-divisor graph} of a ring $R$ with unity, denoted by $\Gamma'(R)$, is an undirected simple graph whose vertex set is  the set of all non-zero and non-unit elements of $R$ and two distinct vertices $x$ and $y$  are adjacent if and only if $x \notin Ry$ and $y \notin Rx$. Together with the relation between the zero-divisor graph and the cozero-divisor graph, Afkhami \emph{et al.} \cite{afkhami2011cozero} studied the basic graph-theoretic properties including completeness, girth, clique number, etc. Then Kavitha \emph{et al.} \cite{kavitha2017genus} determined all the non-local commutative rings $R$ such that $g(\Gamma'(R)) = 1$. For more results on cozero-divisor graphs of rings, we refer the reader to \cite{afkhami2012cozero, afkhami2012planar, afkhami2013cozero} and references therein. Further, Wilkens \emph{et al.} \cite{wilkens2011reduced} introduced the reduced cozero-divisor graph of commutative rings. The \emph{reduced cozero-divisor graph} of a ring $R$, denoted by $\Gamma_r(R)$, is the simple undirected graph with vertex set as the set of all nontrivial principal ideals of $R$ and two distinct vertices $(a)$ and $(b)$ are adjacent  if and only if $(a) \not\subset (b)$ and $(b) \not\subset (a)$. Kavitha \emph{et al.} \cite{kavitha2017genus} determined all the finite non-local commutative rings whose reduced cozero-divisor graph is planar. Further, Jesili \emph{et al.} \cite{jesili2022genus} classified all the non-local commutative rings $R$ such that $g(\Gamma_r(R)) = 1$. The classification of all the non-local commutative rings whose reduced cozero-divisor graphs and cozero-divisor graphs are of genus two has not been investigated yet. In this paper, we characterize all the finite non-local commutative rings for which each of the graph $\Gamma'(R)$ and $\Gamma_r(R)$, is of genus two. The next two theorems are main results of this manuscript.

\begin{theorem}\label{redcozerogenus2}
 Let $R \cong R_1 \times R_2 \times \cdots \times R_n$ $(n \ge 2)$ be a non-local commutative ring, where each $R_i$ is a local ring with maximal ideal $\mathcal{M}_i$ and let $\eta_{i}$ be the nilpotent index of $\mathcal{M}_i$. Then $g(\Gamma_r(R)) = 2$ if and only if either $R \cong R_1 \times F_2$, where $\mathcal{M}_1$ is a principal ideal and $\eta_{1} =7$, or $R \cong R_1 \times R_2$, where both $\mathcal{M}_1$ and $\mathcal{M}_2$ are principal ideals such that $\eta_{1} =4$ and $\eta_{2} =2$.  
\end{theorem}

\begin{theorem}\label{genusofR1R2}
   Let $R \cong R_1 \times R_2 \times \cdots \times R_n$ $(n \ge 2)$ be a non-local commutative ring. Then $g(\Gamma'(R)) = 2$ if and only if $R$ is isomorphic to one of the following $13$ rings: 
  \[ \mathbb{F}_4 \times \mathbb{F}_8, \ \mathbb{F}_4 \times \mathbb{F}_9, \ \mathbb{F}_4 \times \mathbb{Z}_{11}, \ \mathbb{Z}_5 \times \mathbb{Z}_7, \ \mathbb{Z}_4 \times \mathbb{Z}_4, \ \mathbb{Z}_4 \times \dfrac{\mathbb{Z}_2[x]}{( x^2 )}, \ \dfrac{\mathbb{Z}_2[x]}{( x^2 )} \times \dfrac{\mathbb{Z}_2[x]}{( x^2 )}, \ \mathbb{Z}_4 \times \mathbb{F}_8, \ \dfrac{\mathbb{Z}_2[x]}{( x^2 )}\times \mathbb{F}_8, \ \mathbb{Z}_4 \times \mathbb{F}_9, \]
  \[ \dfrac{\mathbb{Z}_2[x]}{( x^2 )} \times \mathbb{F}_9, \ \dfrac{\mathbb{Z}_2[x]}{( x^2 )}\times \mathbb{Z}_{11}, \ \mathbb{Z}_4 \times \mathbb{Z}_{11}.
  \]
\end{theorem}

In the following section, we discuss some basic definitions and related results which we need to understand the remaining part of the present paper.

 \section{preliminaries}
In this section, first we recall the graph theoretic notions from  \cite{westgraph} and topological graph theory from \cite{white1985graphs}. 
A \emph{graph} $\Gamma$ is a pair  $\Gamma = (V, E)$, where $V(\Gamma)$ and $E(\Gamma)$ are the set of vertices and edges of $\Gamma$, respectively. Two distinct vertices $u_1, u_2$ are $\mathit{adjacent}$, denoted by $u_1 \sim u_2$, if there is an edge between $u_1$ and $u_2$. Otherwise, we write it as $u_1 \nsim u_2$. Let $\Gamma$ be a graph. A \emph{subgraph}  $\Gamma'$ of $\Gamma$ is the graph such that $V(\Gamma') \subseteq V(\Gamma)$ and $E(\Gamma') \subseteq E(\Gamma)$. If $X \subseteq V(\Gamma)$, then the subgraph $\Gamma(X)$ of $\Gamma$ induced by $X$, is the graph with vertex set $X$ and two vertices of $\Gamma(X)$ are adjacent if and only if they are adjacent in $\Gamma$. A graph $\Gamma$ is said to be $complete$ if every two distinct vertices are adjacent. The complete graph on $n$ vertices is denoted by $K_n$.

A graph $\Gamma$ is said to be a \emph{bipartite graph} if vertex set of $\Gamma$ can be divided into two disjoint sets $U$ and $V$ such that every edge connects a vertex in $U$ to a vertex in $V$ and no two vertices of $U$(or $V$) are adjacent. A graph $\Gamma$ is said to be a \emph{complete bipartite graph} if the vertex $V(\Gamma)$ can be partitioned into two nonempty sets $A$ and $B$ such that two distinct vertices are adjacent if and only if they belong to different sets. Moreover, if $|A|= m$ and $|B| = n$ then, we denote it by $K_{m,n}$.
 The \emph{subdivision} of the edge $(u,v)$ of a graph $\Gamma$ is the deletion of the $(u,v)$ from $\Gamma$ and the addition of two edges $(u,w)$ and $(w,v)$ along with a new vertex $w$. A graph obtained from $\Gamma$ by a sequence of edge subdivisions is called a \emph{subdivision of $\Gamma$}. The \emph{Contraction} of an edge $(u,v)$, denoted by $[u,v]$, is an operation that delete the edge $(u,v)$ and merges the vertices $u$ and $v$ in the graph. An undirected graph $H$ is a \emph{minor} of another undirected graph $\Gamma$ if $H$ can be obtained from $\Gamma$ by contracting or deleting some edges. Two graphs are said to be \emph{homeomorphic} if both can be obtained from the same graph by subdivision or contracting of edges.

A compact connected topological space  such that each point has a neighbourhood homeomorphic to an open disc is called a surface. Let $\mathbb{S}_{g}$ be the orientable surface with $g$ handles, where $g$ is a non-negative integer. The genus $g(\Gamma)$ of a graph $\Gamma$ is the minimum integer $g$ such that the graph can be embedded in $\mathbb{S}_{g}$, i.e. the graph $\Gamma$ can be drawn into a surface $\mathbb{S}_{g}$ without edge crossing. Note that the graphs having genus $0$ are planar, and graphs having genus $1$ are toroidal. The following results are useful in the sequel, and we will use them frequently without refering to it.

\begin{proposition}{\cite[Ringel and Youngs]{white1985graphs}}
\label{genus}
 Let $m, n$ be a positive integer. Then \\
${\rm(i)}~~ g(K_n) = \left\lceil \frac{(n-3)(n-4)}{12} \right\rceil ~~~~~\text{if} ~~~n \geq 3$\\
${\rm(ii)}~~ g(K_{m,n}) =  \left\lceil \frac{(m-2)(n-2)}{4}  \right\rceil~~ \textit{if}~~ m, n\geq 2$.
 \end{proposition}
 
 \begin{lemma}{\cite[Theorem 5.14]{white1985graphs}}\label{eulerformulagenus}
 Let $\Gamma$ be a connected graph, with a 2-cell embedding in $\mathbb{S}_{g}$. Then $v - e + f = 2 - 2g$, where $v, e$ and $f$ are the number of vertices, edges and faces embedded in $\mathbb{S}_{g}$, respectively and $g$ is the genus of the surface of a graph embedded. 
\end{lemma}

\begin{lemma}\cite{white2001graphs}\label{genusofblocks}
The genus of a connected graph $\Gamma$ is the sum of the genera of its blocks.
\end{lemma}

 Throughout the paper, $R$ denotes the finite non-local commutative ring with unity. For basic definitions of ring theory, we refer the reader to \cite{atiyah1994introduction}. We denote $\mathcal{I}^*(R)$ by the set of all non-zero proper ideals of $R$ and $F_i$ by a field. The field of order $m$ is denoted by $\mathbb{F}_m$. The set $U(R)$ denotes all the units of $R$. We write $I^*$ by the set $I \setminus \{0\}$ of $R$ except additive identity.  The \emph{nilpotent index} of an ideal $I$ of $R$ is the smallest positive integer $n$ such that $I^n = 0$. The ideal generated by the elements $a_1, a_2, \ldots, a_k$ $(k \ge 1)$ of the ring $R$ is denoted by $(a_1, a_2, \ldots, a_k)$. A ring $R$ is called \emph{local} if it has a unique maximal ideal. Let $R$ be a non-local  commutative ring. 
 By the structural theorem \cite{atiyah1994introduction},  $R$ is uniquely (up to isomorphism) a finite direct product of  local rings $R_i$ that is, $R \cong R_1 \times R_2 \times \cdots \times R_n$, where $n \geq 2$.  The following results on the planarity and genus $1$ of the graphs under consideration in this paper, will be useful to characterize all the rings whose (reduced) cozero-divisor graphs are of genus $2$. 
 


\begin{theorem}{\cite[Theorem 8]{jesili2022genus}}\label{redcozerogenuslocal}
Let $R \cong R_1 \times R_2 \times \cdots \times R_n$ be a commutative
ring with identity where each $R_i$ is a local ring with maximal ideal
$\mathcal{M}_i \neq 0$ and $ n \ge 2$. Let $\eta_i$ be the nilpotent index of $\mathcal{M}_i$ for $i = 1, 2$. Then $g(\Gamma_r(R))=1$ if and only if $R$ satisfies the following conditions:

\begin{enumerate}[(1)]
    \item $n = 2$;
    \item  $\mathcal{M}_1 =(a_1)$ and  $\mathcal{M}_2 =(b_1)$	 for some $a_1 \in R_1$, $b_1 \in R_2$ and $2 \le \eta_1$, $\eta_2 \le 3$;
     \begin{enumerate}[\rm(i)]
     \item if $\eta_1 = 3$ and $\eta_2 = 2$, then $\mathcal{M}_1$ and $\mathcal{M}_1^2$ are the only non-trivial principal ideals in $R_1$ and $\mathcal{M}_2$ is the only non-trivial principal ideal in $R_2$.
     \item if $\eta_1 = 2$ and $\eta_2 = 3$, then $\mathcal{M}_1$ is the only non-trivial principal ideal in $R_1$ and $\mathcal{M}_2$ and $\mathcal{M}_2^2$ are the only non-trivial principal ideals in $R_2$.
     \end{enumerate}
\end{enumerate}
\end{theorem}

\begin{theorem}{\cite[Theorem 9]{jesili2022genus}}\label{redcozerogenuslocalfield}
For integers $ n, m \ge 1$, let $R \cong R_1 \times \cdots \times R_n \times  F_1 \times \cdots \times F_m$ be a commutative ring with identity where each $R_i$ $(1 \le i \le n)$ is a local ring with maximal ideal $\mathcal{M}_i \neq \{0\}$ and each $F_j$ $(1 \le j \le m)$ is a field. Then $g(\Gamma_r(R))=1$ if and only if $R$ satisfies one of the following conditions:

\begin{enumerate}[(1)]
    \item$ R \cong R_1 \times F_1 \times F_2$ and $\mathcal{M}_1$ is the only non-trivial principal ideal in $R_1$;
    \item $R \cong R_1 \times F_1$ and
     \begin{enumerate}[\rm(i)]
     \item If $\mathcal{M}_1 = (b_1,b_2)$, then $(b_1)$, $(b_2)$, $(b_1b_2)$ and $(b_1 + b_2)$ are the only nontrivial principal ideals of $R_1$.
     \item $\mathcal{M}_1 = (b_1)$ is a principal ideal in $R_1$ with nilpotency $\eta_1 = 5$ or $6$;
\begin{enumerate}[(a)]
    \item If $\eta_1 = 5$, then $\mathcal{M}_1$, $\mathcal{M}_1^2$, $\mathcal{M}_1^3$ and $\mathcal{M}_1^4$ are the only non-trivial principal ideals of $R_1$.
    \item If $\eta_1 = 6$, then $\mathcal{M}_1$, $\mathcal{M}_1^2$, $\mathcal{M}_1^3$, $\mathcal{M}_1^4$ and $\mathcal{M}_1^5$ are the only non-trivial principal ideals of $R_1$.
\end{enumerate}
\end{enumerate}
\end{enumerate}
\end{theorem}

\begin{theorem}{\cite[Theorem 2.5]{afkhami2012planar}}\label{cozeroplanarall}
Let R be a finite commutative non-local ring. Then $\Gamma'(R)$ is planar if and only if $R$ is isomorphic to one of the following rings: $\mathbb{Z}_2 \times \mathbb{Z}_2 \times \mathbb{Z}_2 $, $ \mathbb{Z}_2 \times F_2$, $\mathbb{Z}_2 \times\mathbb{Z}_4$, $\mathbb{Z}_2 \times \dfrac{\mathbb{Z}_2[x]}{(x^2)}$, $ \mathbb{Z}_3 \times F_2$, $\mathbb{Z}_3 \times\mathbb{Z}_4$ or $\mathbb{Z}_3 \times \dfrac{\mathbb{Z}_2[x]}{(x^2)}$, where $F_2$ is a finite field.
\end{theorem}

\begin{theorem}{\cite[Theorem 2.1]{kavitha2017genus}}\label{cozerogenusfields}
Let $R \cong F_1 \times \cdots \times F_n$ be a finite commutative ring with identity, where each $F_j$ is a field and $ n \ge 2$.
Then $g(\Gamma'(R)) = 1$ if and only if $R$ is isomorphic to one of the following rings: $\mathbb{F}_4 \times \mathbb{F}_4$, $\mathbb{F}_4 \times \mathbb{Z}_5$, $\mathbb{Z}_5 \times \mathbb{Z}_5$, $\mathbb{F}_4 \times \mathbb{Z}_7$ or $\mathbb{Z}_3 \times \mathbb{Z}_2 \times \mathbb{Z}_2$.
\end{theorem}

\begin{theorem}{\cite[Theorem 2.3]{kavitha2017genus}}\label{cozerogenuslocalfield}
Let $R = R_1 \times \cdots \times R_n \times F_1 \times \cdots \times F_m$ be a commutative ring with identity, where each $R_i$ is a local ring with maximal ideal $\mathcal{M}_i \neq \{0 \}$, $F_j$ is a field and $n, m \ge 1$. Then $g(\Gamma'(R)) = 1$ if and only if $R$ is isomorphic to one of the following rings: 
$\mathbb{Z}_4 \times  \mathbb{F}_4$, $\dfrac{\mathbb{Z}_2[x]}{(x^2)} \times  \mathbb{F}_4 $, $\mathbb{Z}_4 \times  \mathbb{Z}_5$, $\dfrac{\mathbb{Z}_2[x]}{(x^2)} \times \mathbb{Z}_5$, $\mathbb{Z}_4 \times  \mathbb{Z}_7$, $\dfrac{\mathbb{Z}_2[x]}{(x^2)} \times \mathbb{Z}_7$, $\mathbb{Z}_9 \times  \mathbb{Z}_2$, $\dfrac{\mathbb{Z}_3[x]}{(x^2)} \times \mathbb{Z}_2$, $\mathbb{Z}_8 \times  \mathbb{Z}_2$, $\dfrac{\mathbb{Z}_2[x]}{(x^3)} \times \mathbb{Z}_2$, $\dfrac{\mathbb{Z}_4[x]}{(2x, x^2-2)} \times \mathbb{Z}_2$, $\dfrac{\mathbb{Z}_4[x]}{(2,x)^2} \times \mathbb{Z}_2$, or $\dfrac{\mathbb{Z}_2[x,y]}{(x,y)^2} \times \mathbb{Z}_2$.
\end{theorem}

\begin{theorem}{\cite[Theorem 3.1]{kavitha2017genus}}\label{redcozeroplanarfields}
Let $R \cong F_1 \times \cdots \times F_n$ be a finite commutative ring with identity, where each $F_j$ is a field and $ n \ge 2$. Then $\Gamma_r(R)$ is planar if and only if R is isomorphic to one of the following rings: $F_1 \times F_2 \times F_3$ or $F_1 \times F_2$.
\end{theorem}

\begin{theorem}{\cite[Theorem 3.2]{kavitha2017genus}}\label{redcozeroplanarlocal}
Let $R \cong R_1 \times \cdots \times R_n$ be a commutative ring with identity $1$,  where each $R_i$ is a local ring with maximal ideal $\mathcal{M}_i \neq \{0 \}$ and $n \ge 2$. Then $\Gamma_r(R)$ is planar if and only if $R \cong R_1 \times R_2$ such that $\mathcal{M}_i$ is the only non-zero principal ideal in $R_i$.
\end{theorem}
 
\section{Proof of the main theorems}

In order to prove the Theorem \ref{redcozerogenus2}, first we establish the following lemmas.

\begin{lemma}\label{genusgreaterthan2}
Let $R \cong F_1 \times F_2 \times F_3 \times F_4$. Then $g(\Gamma_r(R)) \geq 3$.   
\end{lemma}

\begin{proof}
 Consider the vertices $x_1=(0) \times (1) \times (0) \times (1)$, $x_2= (0) \times (1) \times (0) \times (0) $, $x_3= (0)\times (1) \times (1) \times (1) $, $x_4= (0) \times (1) \times (1) \times (0) $, $x_5= (0) \times (0) \times (1) \times (1) $, $x_6 = (1) \times (0) \times (0) \times (1) $, $x_7= (1) \times (0) \times (0) \times (0) $, $x_8= (1) \times (0) \times (1) \times (1) $, $x_9= (1) \times (0) \times (1) \times (0)  $, $x_{10}= (1) \times (1) \times (0) \times (0) $, $x_{11}= (0) \times (0) \times (1) \times (0) $ and $x_{12}= (1) \times (1) \times (0) \times (1)$ of $g(\Gamma_r(R))$. Then the graph induced by the set $A = \{x_1, x_2, x_3, x_4, x_5, x_6, x_7, x_8, x_9, x_{10}, x_{11}, x_{12} \}$ contains a subgraph homeomorphic to $K_{5,5}$ (see Figure \ref{genusK55}). By Proposition \ref{genus}, we have $g(\Gamma_r(R)) \geq 3$.
\begin{figure}[h!]
    \centering
    \includegraphics[width=0.75 \textwidth]{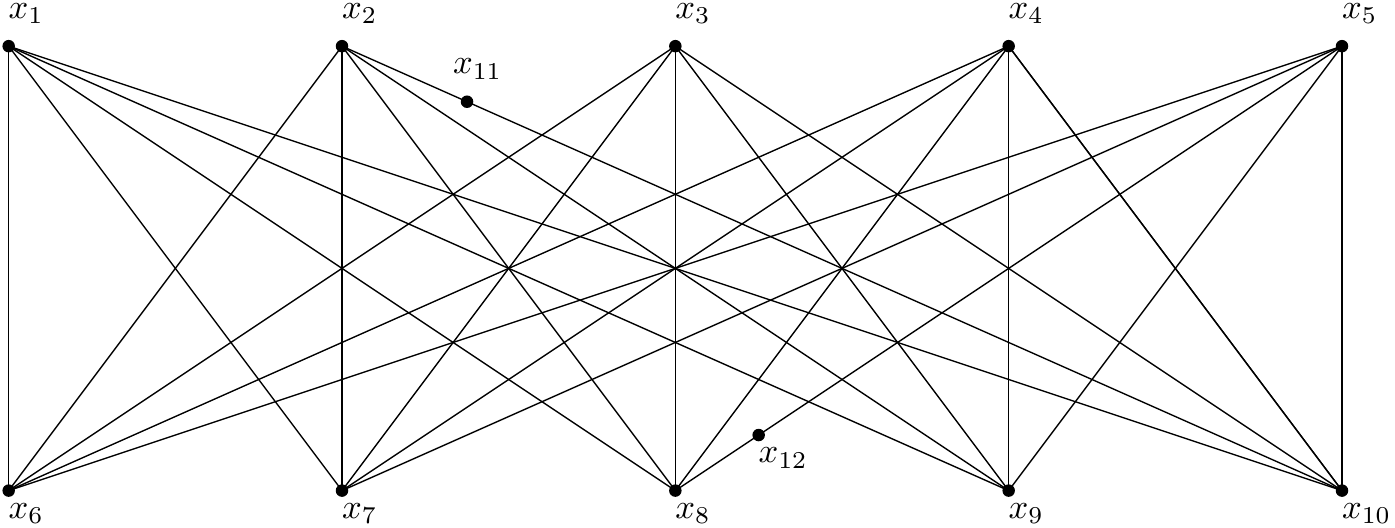}
    \caption{Subgraph of $\Gamma_r(A)$ homeomorphic to $K_{5,5}$.}
    \label{genusK55}
    \end{figure}
\end{proof}

\begin{lemma}\label{genusnotequalto2}
Let $R \cong R_1 \times R_2 \times \cdots \times R_n$ $(n \geq 3)$ be a non-local commutative ring. Then $g(\Gamma_r(R)) \neq 2$. 
\end{lemma}

\begin{proof}
If $n \geq 4$, then Lemma \ref{genusgreaterthan2} follows that $g(\Gamma_r(R)) \geq 3$. Now let $n = 3$. Then $R \cong R_1 \times R_2 \times R_3$. On contrary assume that $g(\Gamma_r(R)) = 2$. If each $R_i$ is a field, then by Theorem \ref{redcozeroplanarfields}, we obtain $g(\Gamma_r(R)) = 0$. Let $R_1$ be not a field and $R_2$, $R_3$ be fields.  Suppose that $\mathcal{M}_1 = (a_1, a_2, \ldots, a_k)$, where $k \geq 2$, is a maximal ideal of $R_1$. Then consider the sets $X =  \{(a_1) \times R_2 \times R_3,\ (a_1 + a_2) \times (0) \times R_3, \ (a_1)\times (0)\times R_3, \ (a_1) \times R_2 \times (0), \ (a_1)\times (0) \times (0) \}$ and $Y = \{(a_2) \times R_2 \times R_3, \ (a_1 + a_2) \times R_2 \times (0), \ (a_2)\times (0)\times R_3, \ (a_2) \times R_2 \times (0), \ (a_2)\times (0) \times (0)\}$. The subgraph induced by the set $X \cup Y$ contains a subgraph isomorphic to $K_{5,5}$ with partition sets $X$ and $Y$. Therefore, by Proposition \ref{genus}, we obtain $g(\Gamma_r(R)) \geq 3$. It follows that $\mathcal{M}_1 = (a_1)$. If $R_1$ has only one nontrivial ideal $\mathcal{M}_1$, then by Theorem \ref{redcozerogenuslocalfield}, the graph $\Gamma_r(R)$ has genus $1$. We may now suppose that $\mathcal{M}_1 = (a_1)$ is a principal ideal with nilpotent index at least three. Then $\Gamma_r(R)$ contains a subgraph homeomorphic to $K_{5,5}$ (see Figure \ref{subdivision of K55 fig2}). By Proposition \ref{genus}, we have $g(\Gamma_r(R)) \geq 3$.
\begin{figure}[h!]
\centering
\includegraphics[width=0.75 \textwidth]{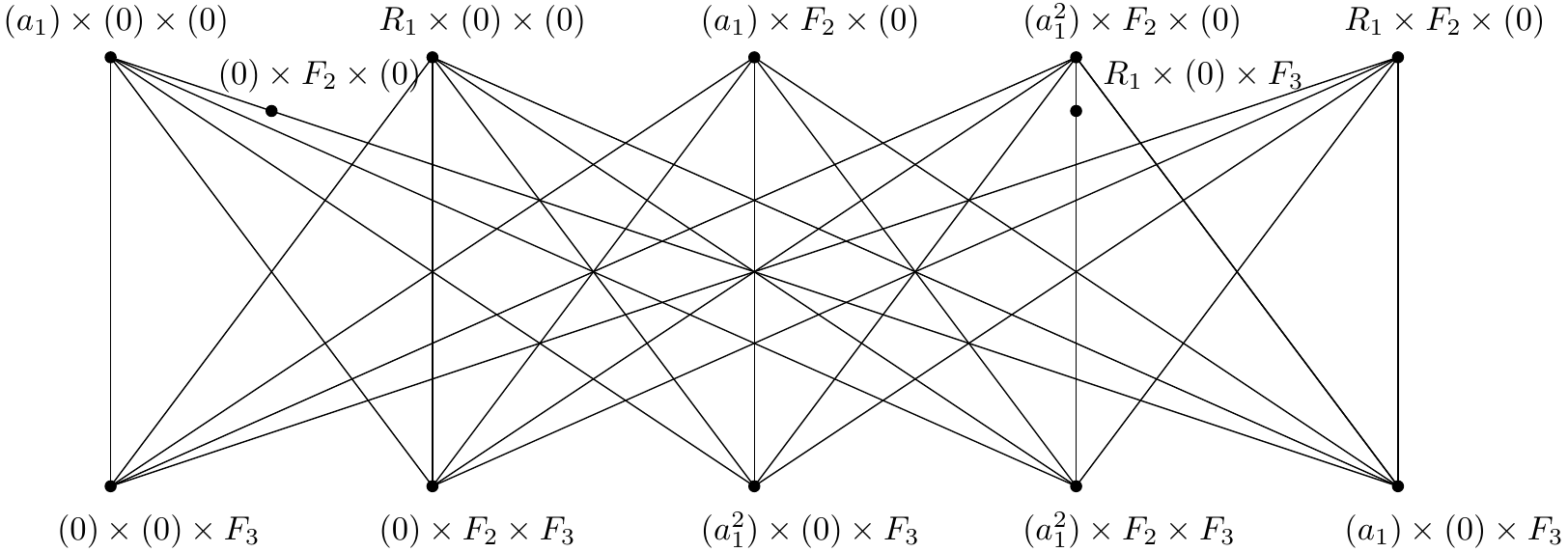}
\caption{ \centering Subgraph of $\Gamma_r(R_1 \times F_2 \times F_3 )$, where $\eta_1 \ge 3$.}
\label{subdivision of K55 fig2}
\end{figure}

Further, assume that either $R_2$ or $R_3$ is not a field. Without loss of generality, suppose that $R_2$ is not a field and $(b_1)$ is a nontrivial ideal of $R_2$. Then $\Gamma_r(R)$ contains a subgraph homeomorphic to $K_{5,5}$ (see Figure \ref{subdivision of K55 fig1}) and so by Proposition \ref{genus}, we get $g(\Gamma_r(R)) \geq 3$.
\begin{figure}[h!]
\centering
\includegraphics[width=0.75 \textwidth]{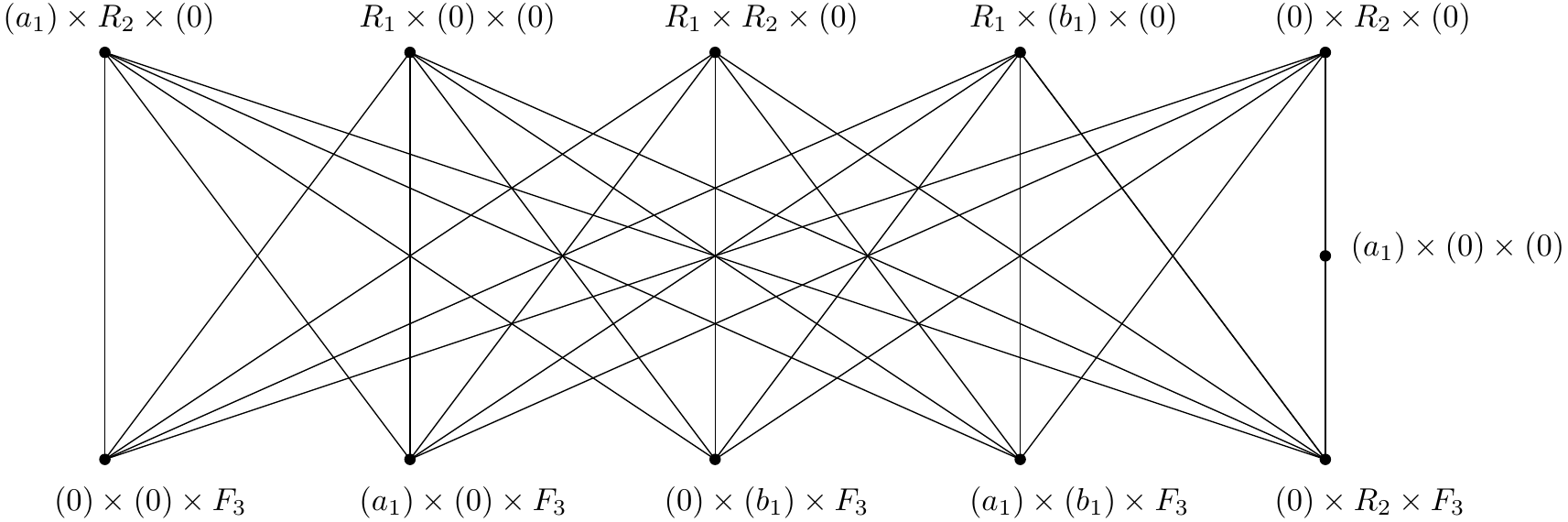}
\caption{ \centering A subgraph of $\Gamma_r(R_1 \times R_2 \times F_3 )$, where $\mathcal{M}_1 = (a_1)$.}
\label{subdivision of K55 fig1}
\end{figure}
This completes our proof.
\end{proof}

\noindent\textbf{Proof of Theorem \ref{redcozerogenus2}:}
First suppose that $g(\Gamma_r(R)) = 2$. By Lemma \ref{genusnotequalto2}, we have $R \cong R_1 \times R_2$, where $R_1$ and $ R_2$ are local rings having maximal ideals $\mathcal{M}_1$ and $\mathcal{M}_2$, respectively. Suppose that one of $\mathcal{M}_i$ is not principal. Without loss of generality, assume that $\mathcal{M}_1 = (a_1, a_2)$. 
 Let $R_2$ be a field. If $a_1^{2} \neq 0$ and $a_2^{2} \neq  0$, then $\Gamma_r(R_1 \times F_2)$ contains a subgraph homeomorphic to $K_{5,5}$ (see Figure \ref{reducedcozeroK55fig5}), a contradiction.
  \begin{figure}[h!]
\centering
\includegraphics[width=0.75 \textwidth]{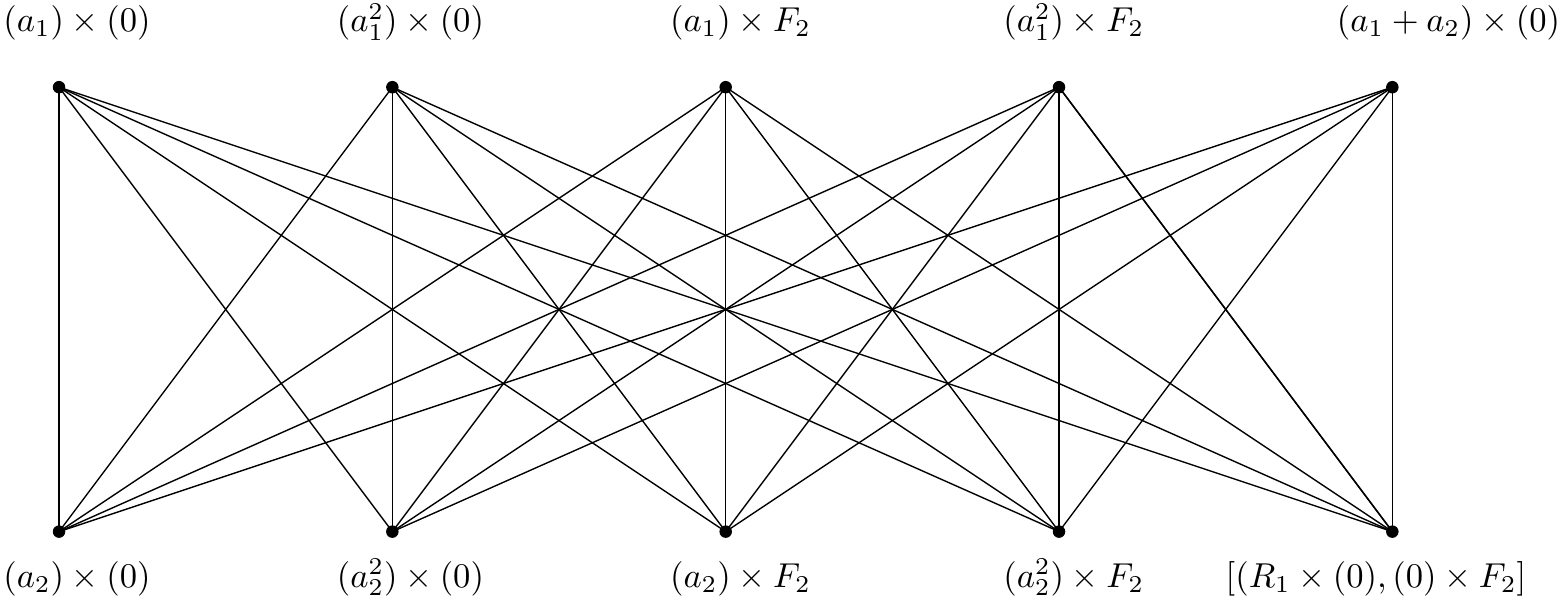}
\caption{ \centering A subgraph of $\Gamma_r(R_1 \times F_2)$, where $\mathcal{M}_1 = (a_1, a_2)$ such that $a_1^2, a_2^2 \neq 0$.}
\label{reducedcozeroK55fig5}
\end{figure}
 If $a_1^{2} \neq 0$ but $a_2^{2} = 0$, then note that $\Gamma_r(R_1 \times F_2)$ contains a subgraph homeomorphic to $K_{5,5}$ (see Figure \ref{reducedcozeroK55fig6}), which is not possible.
 \begin{figure}[h!]
\centering
\includegraphics[width=0.75 \textwidth]{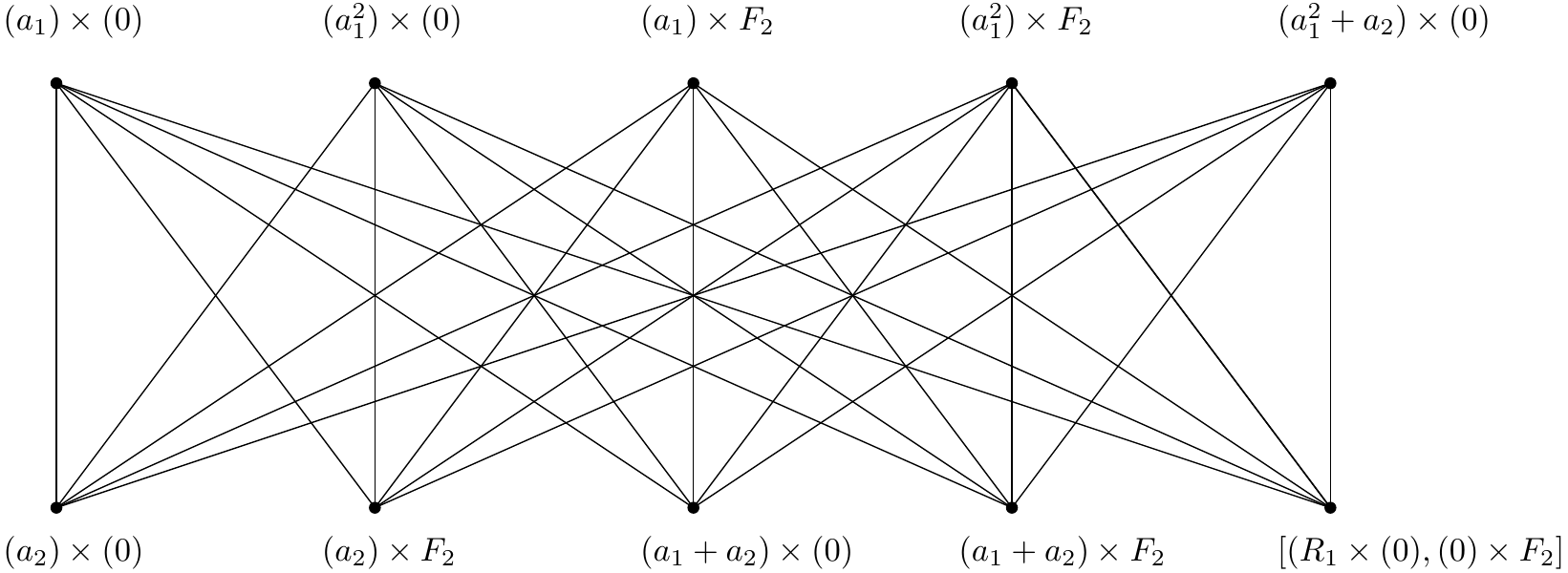}
\caption{ \centering A subgraph of $\Gamma_r(R_1 \times F_2)$, where $\mathcal{M}_1 = (a_1, a_2)$ such that $a_1^2 \neq 0$ but $a_2^2 = 0$.}
\label{reducedcozeroK55fig6}
\end{figure}

 Next, let $a_1^{2} = 0 = a_2^{2}$. If $|U(R_1)| = 1$, then $(a_1)$, $(a_2)$, $(a_1 + a_2)$ and $(a_1a_2)$ are the only nontrivial principal ideals of $R_1$. By Theorem \ref{redcozerogenuslocalfield}, we have $g(\Gamma_r(R)) = 1$. Next, assume that $|U(R_1)| = 2$ and $1, \alpha \in U(R_1) $. Consider the nontrivial principal ideals $(a_1)$, $(a_2)$, $(a_3)$, and $(a_4)$ of $R_1$, where $a_3 = a_1 + a_2$ and $a_4 = a_1 + \alpha a_2$. Note that $(a_i) \not\subset (a_j)$ for distinct $i,j \in \{1, 2, 3, 4\}$. Then by Figure \ref{subdivision of 2K5}, note that $\Gamma_r(R)$ contains a subgraph $H$ isomorphic to $K_{5}$ and so $g(\Gamma_r(H)) = 1$.
 \begin{figure}[h!]
\centering
\includegraphics[width=0.5 \textwidth]{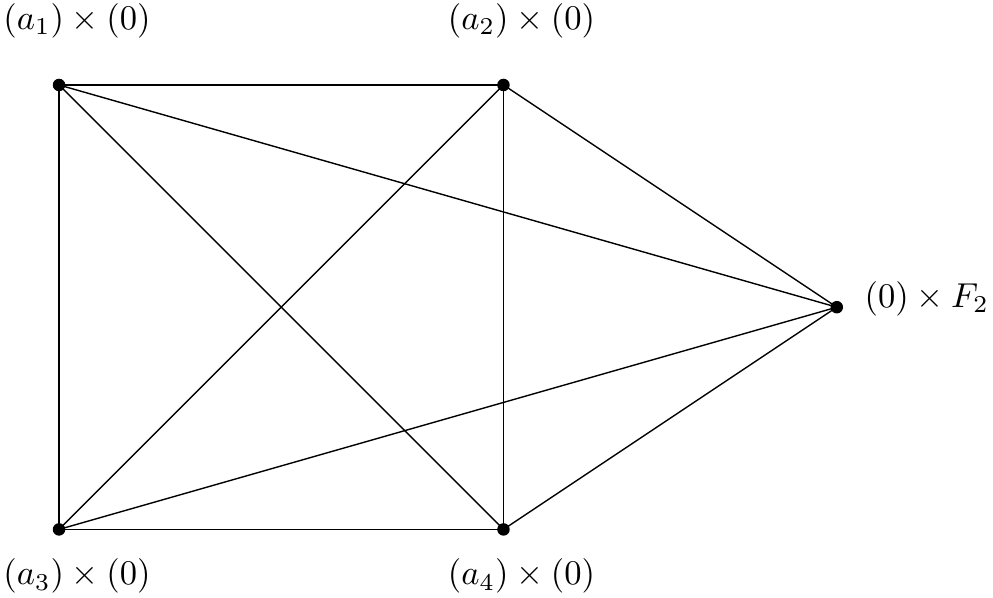}
\caption{A subgraph $H$ of $\Gamma_r(R_1 \times F_2 )$, where $\mathcal{M}_1 = (a_1, a_2)$.}
\label{subdivision of 2K5}
\end{figure}
It implies that $g(\Gamma_r(R)) \geq 1$. Suppose that $g(\Gamma_r(R)) = 1$. Now to embed $\Gamma_r(R)$ in $\mathbb{S}_1$ through $H$, first we insert the vertices $(a_1) \times F_2$, $(a_2) \times F_2$ and $(a_3) \times F_2$ in an embedding of $H$ in $\mathbb{S}_1$. Since the vertices $(a_1) \times F_2$, $(a_2) \times F_2$ and $(a_3) \times F_2$ are adjacent with each other, therefore they must be inserted in the same face $F$. Note that the vertex $(a_1) \times F_2$ is adjacent with $(a_2) \times (0)$, $(a_3) \times (0)$ and $(a_4) \times (0)$. Also, $(a_2) \times F_2$ is adjacent with $(a_1) \times (0)$, $(a_3) \times (0)$ and $(a_4) \times (0)$. Furthermore, $(a_3) \times F_2$ is adjacent with $(a_1) \times (0)$, $(a_2) \times (0)$ and $(a_4) \times (0)$. It follows that the face $F$ must contain the vertices $(a_1) \times (0)$, $(a_2) \times (0)$, $(a_3) \times (0)$ and $(a_4) \times (0)$. Consequently, insertion of the vertices $(a_1) \times F_2$, $(a_2) \times F_2$, $(a_3) \times F_2$ and their incident edges in $F$ is not possible without edge crossings, a contradiction. Now let $H'$ be the subgraph induced by the vertex set $V(H) \cup \{ (a_1) \times F_2$, \ $(a_2) \times F_2$, \ $(a_3) \times F_2\}$. It implies that $g(\Gamma_r(H')) \geq 2$ and so $g(\Gamma_r(R)) \geq 2$. Suppose $g(\Gamma_r(H')) = 2$. Using $H'$, to embed $\Gamma_r(R)$ in $\mathbb{S}_2$ first we insert the vertices $R_1 \times (0)$, $(a_4) \times F_2$ and their incident edges in an embedding of $H'$ in $\mathbb{S}_2$. Since both the vertices $R_1 \times (0)$ and $(a_4) \times F_2$ are adjacent, they must be inserted in the same face $F'$. Note that both the vertices $R_1 \times (0)$ and $(a_4) \times F_2$ are adjacent with  $(a_1) \times F_2$, $(a_2) \times F_2$ and $(a_3) \times F_2$. It implies that $F'$ must contain the vertices $(a_1) \times F_2$, $(a_2) \times F_2$ and $(a_3) \times F_2$, which leads to an edge crossing, a contradiction. It follows that $g(\Gamma_r(R_1 \times R_2)) \geq 3$. Therefore, for  $g(\Gamma_r(R)) = 2$, we must have $R \cong R_1 \times R_2$ such that $\mathcal{M}_1$ and $\mathcal{M}_2$ are principal ideals of $R_1$ and $R_2$, respectively.   
 
 First suppose that $R_2$ is a field and $\mathcal{M}_1 = (a_1)$, where $a_1 \neq 0$. If $\eta_1 =8$, then by Figure \ref{subdivision of 2K33 of nilpotency8}, note that $\Gamma_r(R)$ contains a subgraph $G$ isomorphic to $K_{3,3}$.
 \begin{figure}[h!]
\centering
\includegraphics[width=0.5 \textwidth]{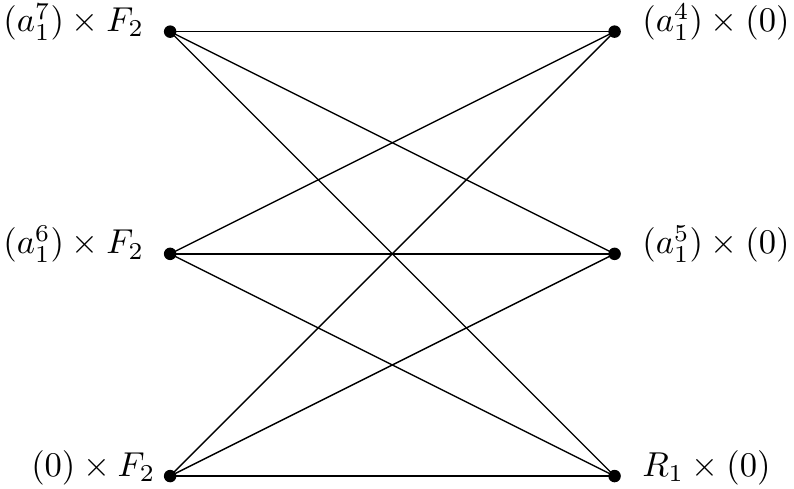}
\caption{Subgraph $G$ of $\Gamma_r(R_1 \times F_2 )$, where $\eta_1 =8$.}
\label{subdivision of 2K33 of nilpotency8}
\end{figure}
It follows that $g(\Gamma_r(G)) = 1$ and so $g(\Gamma_r(R)) \ge 1$. Suppose that $g(\Gamma_r(R)) = 1$. Now to embed $\Gamma_r(R)$ in $\mathbb{S}_1$ through $G$, first we insert the vertices $(a_1^{3}) \times (0)$, $(a_1^{4}) \times F_2$, $(a_1^{5}) \times F_2$ and their incident edges in an embedding of $G$ in $\mathbb{S}_1$. Since $(a_1^{4}) \times F_2 \sim (a_1^{3}) \times (0) \sim (a_1^{5}) \times F_2$, they must be inserted in the same face $F''$. Note that the vertex $(a_1^{3}) \times (0)$ is adjacent with $(a_1^{6}) \times F_2$, $(a_1^{7}) \times F_2$ and $(0) \times F_2$. Also, $(a_1^{4}) \times F_2$ is adjacent with $R_1 \times (0)$. Furthermore, $(a_1^{5}) \times F_2$ is adjacent with $R_1 \times (0)$ and $(a_1^{4}) \times (0)$. It follows that the face $F''$ must contain the vertices $(a_1^{6}) \times F_2$, $(a_1^{7}) \times F_2$, $(0) \times F_2$, $R_1 \times (0)$ and $(a_1^{4}) \times (0)$. Since $G$ is a bipartite graph, $F''$ must be of length $6$ containing all the vertices of $G$. After inserting the vertex $(a_1^{3}) \times (0)$ in $F''$, insertion of the vertices $(a_1^{4}) \times F_2$ and $(a_1^{5}) \times F_2$ is not possible without edge crossings. Let $G'$ be the graph induced by the vertex set $V(G) \cup \{ (a_1^{3}) \times (0), \ (a_1^{4}) \times F_2, \ (a_1^{5}) \times F_2\} $. It implies that $g(\Gamma_r(G')) \ge 2$ and so $g(\Gamma_r(R)) \ge 2$. Suppose $g(\Gamma_r(G')) = 2$. To embed $\Gamma_r(R)$ in $\mathbb{S}_2$, first we insert the vertices $(a_1) \times (0)$, $(a_1^2) \times (0)$, $(a_1^3) \times F_2$ and their incident edges in the embedding of $G'$ in $\mathbb{S}_2$. Since $(a_1) \times (0) \sim (a_1^3) \times F_2 \sim (a_1^2) \times (0)$, they must be inserted in the same face $F'''$. Note that both the vertices $(a_1) \times (0)$ and $(a_1^2) \times (0)$ are adjacent with  $(a_1^6) \times F_2$, $(a_1^7) \times F_2$ and $(0) \times F_2$. It implies that $F'''$ must contain the vertices $(a_1^6) \times F_2$, $(a_1^7) \times F_2$ and $(0) \times F_2$ which leads to an edge crossing, a contradiction. Therefore,  $g(\Gamma_r(R)) \geq 3$, a contradiction.
If $\eta_1 \le 6$, then $g(\Gamma_r(R)) \leq 1$ (cf. \cite{jesili2022genus, kavitha2017genus}), a contradiction. Thus, $R \cong R_1 \times F_2$ such that $\eta_1 =7$.

Further, we may now assume that $R_2$ is not a field. Consider $\mathcal{M}_1 = (a_1)$ and $\mathcal{M}_2 =(b_1)$. Without loss of generality suppose that $\eta_1 = 5$. If $\eta_2 = 2$, then by Figure \ref{reducedcozeroK33nilpotency5,2}, $\Gamma_r(R)$ contains a subgraph $K$ isomorphic to $K_{3,3}$.
\begin{figure}[h!]
\centering
\includegraphics[width=0.5 \textwidth]{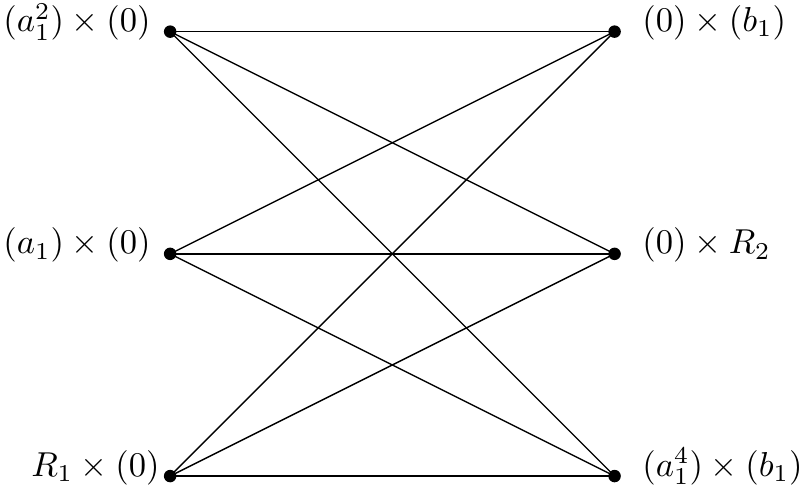}
\caption{Subgraph $K$ of $\Gamma_r(R_1 \times R_2 )$, where $\eta_1 =5$ and $\eta_2 =2$.}
\label{reducedcozeroK33nilpotency5,2}
\end{figure}
Consequently, $g(\Gamma_r(K)) = 1$ and so $g(\Gamma_r(R)) \geq 1$. Suppose that $g(\Gamma_r(R)) = 1$. To embed $\Gamma_r(R)$ in $\mathbb{S}_1$, first we insert the vertices $(a_1^4) \times R_2$, $(a_1^2) \times (b_1)$ and $(a_1^3) \times R_2$ and their incident edges in an embedding of $K$ in $\mathbb{S}_1$. Since  $(a_1^4) \times R_2 \sim (a_1^2) \times (b_1) \sim (a_1^3) \times R_2$, these vertices must be inserted in the same face $F^1$. Note that both the vertices $(a_1^4) \times R_2$ and $(a_1^3) \times R_2$ are adjacent with the vertices $(a_1) \times (0)$, $(a_1^2) \times (0)$ and $R_1 \times (0)$. Therefore, $F^1$ must contain the vertices $(a_1) \times (0)$, $(a_1^2) \times (0)$ and $R_1 \times (0)$. It implies that insertion of the vertices $(a_1^4) \times R_2$ and $(a_1^3) \times R_2$ in $F^1$ is not possible without edge crossings. Now let $K'$ be the subraph of $\Gamma_r(R)$, induced by the vertex set $V(K) \cup \{ (a_1^4) \times R_2, \ (a_1^2) \times (b_1), \ (a_1^3) \times R_2 \}$. It follows that $g(\Gamma_r(K')) \ge 2$ and so $g(\Gamma_r(R)) \ge 2$. Suppose that $g(\Gamma_r(K')) = g(\Gamma_r(R)) = 2$. Now to embed $\Gamma_r(R)$ in $\mathbb{S}_2$, first we insert the vertices $R_1 \times (b_1)$, $(a_1^2) \times R_2$ and $(a_1) \times (b_1)$ and their incident edges in an embedding of $K'$ in $\mathbb{S}_2$. Since $R_1 \times (b_1) \sim (a_1^2) \times R_2 \sim (a_1) \times (b_1)$, they must be inserted in the same face $F^2$. Note that both the vertices $R_1 \times (b_1)$ and $(a_1) \times  (b_1)$ are adjacent with the vertices $(0) \times R_2$, $(a_1^3) \times R_2$ and $(a_1^4) \times R_2$. Therefore, insertion of the vertices $R_1 \times (b_1)$ and $(a_1) \times  (b_1)$ in the same face is not possible without edge crossings. Consequently, $g(\Gamma_r(R)) \ge 3$, a contradiction. It implies that $\eta_i \le 4$ for each $i \in \{ 1,2\}$. 

Without loss of generality assume that $\eta_1 = 4$.
If $\eta_2 = 3$, then consider the sets $U_1 = \{(a_1^{3})\times R_2, \ (a_1^{3}) \times (b_1), \ (0) \times (b_1), \ (a_1^{2})\times (b_1), \ (0)\times R_2\}$ and $V_1 = \{R_1 \times (0), \ R_1 \times (b_1^{2}), \ (a_1) \times (0), \ (a_1) \times (b_1^{2}), \ (a_1^{2}) \times (b_1^{2})\}$. Then the set $U_1 \cup V_1$ induced a subgraph $\gamma$ homeomorphic to $K_{5,5}-\{\big((a_1^{2}) \times (b_1^{2}), \ (a_1^{2})\times (b_1)\big)\}$. Since $\gamma$ contains a subgraph isomorphic to $K_{5,4}$, we have $g(\Gamma_r(\gamma)) \geq 2$ and so $g(\Gamma_r(R)) \geq 2$. Suppose that $g(\Gamma_r(R)) = 2$. Now to embed $\Gamma_r(R)$ in $\mathbb{S}_2$, first we insert the vertices $R_1 \times (b_1)$, $(a_1^2) \times R_2$, $(a_1) \times (b_1)$ and their incident edges in an embedding of $\gamma$ in $\mathbb{S}_2$. Since $R_1 \times (b_1) \sim (a_1^2) \times R_2 \sim (a_1) \times (b_1)$, they must be inserted in the same face $F$. Note that $R_1 \times (b_1)$ is adjacent with $(0) \times R_2$ and $(a_1^3) \times R_2$. Also, $(a_1^2) \times R_2$ is adjacent with $R_1 \times (0)$. Further, $(a_1) \times (b_1)$ is adjacent with $(0) \times R_2$, $(a_1^3) \times R_2$ and $R_1 \times (0)$. It implies that $F$ must contain the vertices $(0) \times R_2$, $(a_1^3) \times R_2$ and $R_1 \times (0)$. After inserting the vertices $R_1 \times (b_1)$ and $(a_1^2) \times R_2$ in $F$, insertion of the vertex $(a_1) \times (b_1)$ is not possible without edge crossings. It implies that $g(\Gamma_r(R)) \geq 3$, which is a contradiction. It follows that $\eta_1 = 4$ and $\eta_2 = 2$. Further, suppose that $\eta_1 = 3$. If $\eta_2 = 3$, then by Figure \ref{k55minus1}, we have a subgraph $H$ of $\Gamma_r(R)$ which is isomorphic to $K_{5,5}-\{\big((a_1^{2}) \times (0), \ (a_1^{2}) \times (b_1)\big), \  \big((a_1^{2}) \times (0), \ (a_1^{2}) \times R_2\big), \ \big((0) \times (b_1^{2}), \ (a_1) \times (b_1^{2})\big), \ \big((0) \times (b_1^{2}), \ R_1 \times (b_1^{2})\big)\}$.
\begin{figure}[h!]
\centering
\includegraphics[width=0.75 \textwidth]{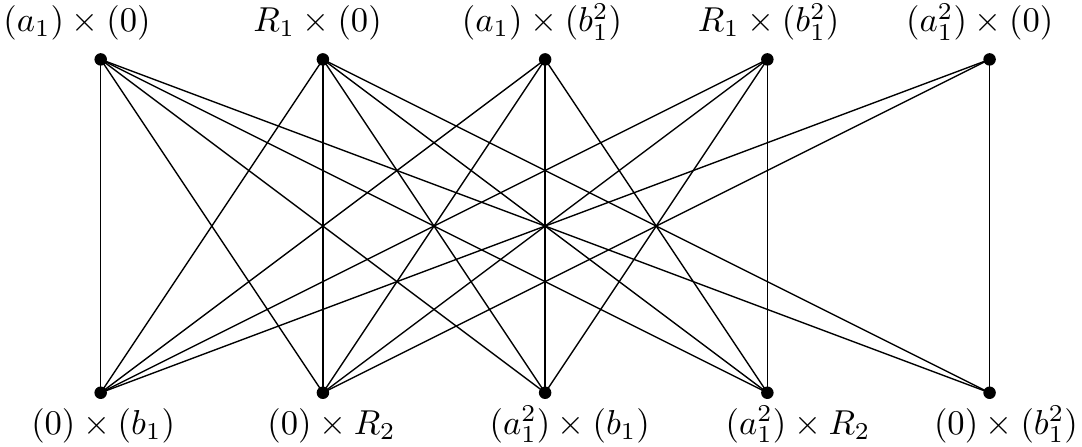}
\caption{Subgraph $H$ of $\Gamma_r(R_1 \times R_2)$, where $\eta_1 =3$ and $\eta_2 =3$. }
\label{k55minus1}
\end{figure}
Observe that $H$ is a bipartite graph and contains a subgraph isomorphic to $K_{4,4}$. Consequently, $g(\Gamma_r(H)) \ge 1$. Supppose that $g(\Gamma_r(H)) =1$. Note that the subgraph $H$ has $10$ vertices and $21$ edges. By Lemma \ref{eulerformulagenus}, we get $f=11$ in the embedding of $H$ in $\mathbb{S}_1$. It implies that $4f > 2e$, which is not possible. It follows that $g(\Gamma_r(R)) \ge g(\Gamma_r(H)) \geq 2$. Assume that $g(\Gamma_r(R)) = 2$. Now to embed the $\Gamma_r(R)$ in $\mathbb{S}_2$, first we insert the vertices $(a_1^{2}) \times (b_1^{2})$, $(a_1) \times R_2$, $R_1 \times (b_1)$, $(a_1) \times (b_1)$ and their respective incident edges. Since $(a_1) \times R_2 \sim R_1 \times (b_1)$, these vertices must be inserted in the same face $F'$. Note that $(a_1) \times R_2$ is adjacent with $R_1 \times (0)$ and $R_1 \times (b_1^{2})$. Also, $R_1 \times (b_1)$ is adjacent with $(0) \times R_2$ and $(a_1^{2})\times R_2$. It implies that the face $F'$ must contain the vertices $R_1 \times (0)$, $R_1 \times (b_1^{2})$, $(0) \times R_2$ and $(a_1^{2})\times R_2$. Notice that the vertex $(a_1) \times (b_1)$ is also adjacent with $R_1 \times (0)$, $R_1 \times (b_1^{2})$, $(0) \times R_2$ and $(a_1^{2})\times R_2$. Since there does not exist any vertex of degree $2$ in the graph $H$, these four vertex cannot be common in two faces. It follows that the vertex $(a_1) \times (b_1)$ must be inserted in $F'$, which is not possible without edge crossings. Therefore, $g(\Gamma_r(R)) \geq 3$, a contradiction. If $\eta_2 = 2$, then by Theorem \ref{redcozerogenuslocal}, we get $g(\Gamma_r(R)) = 1$, a contradiction. We may now assume that $\eta_1 =2 $ and $\eta_2 = 2$. Then by Theorem \ref{redcozeroplanarlocal}, $g(\Gamma_r(R)) = 0$, a contradiction.

The converse follows from Figure \ref{genus2nilpotency7} and Figure \ref{genus2nilpotency4,2}.
\begin{figure}[h!]
\centering
\includegraphics[width=0.6 \textwidth]{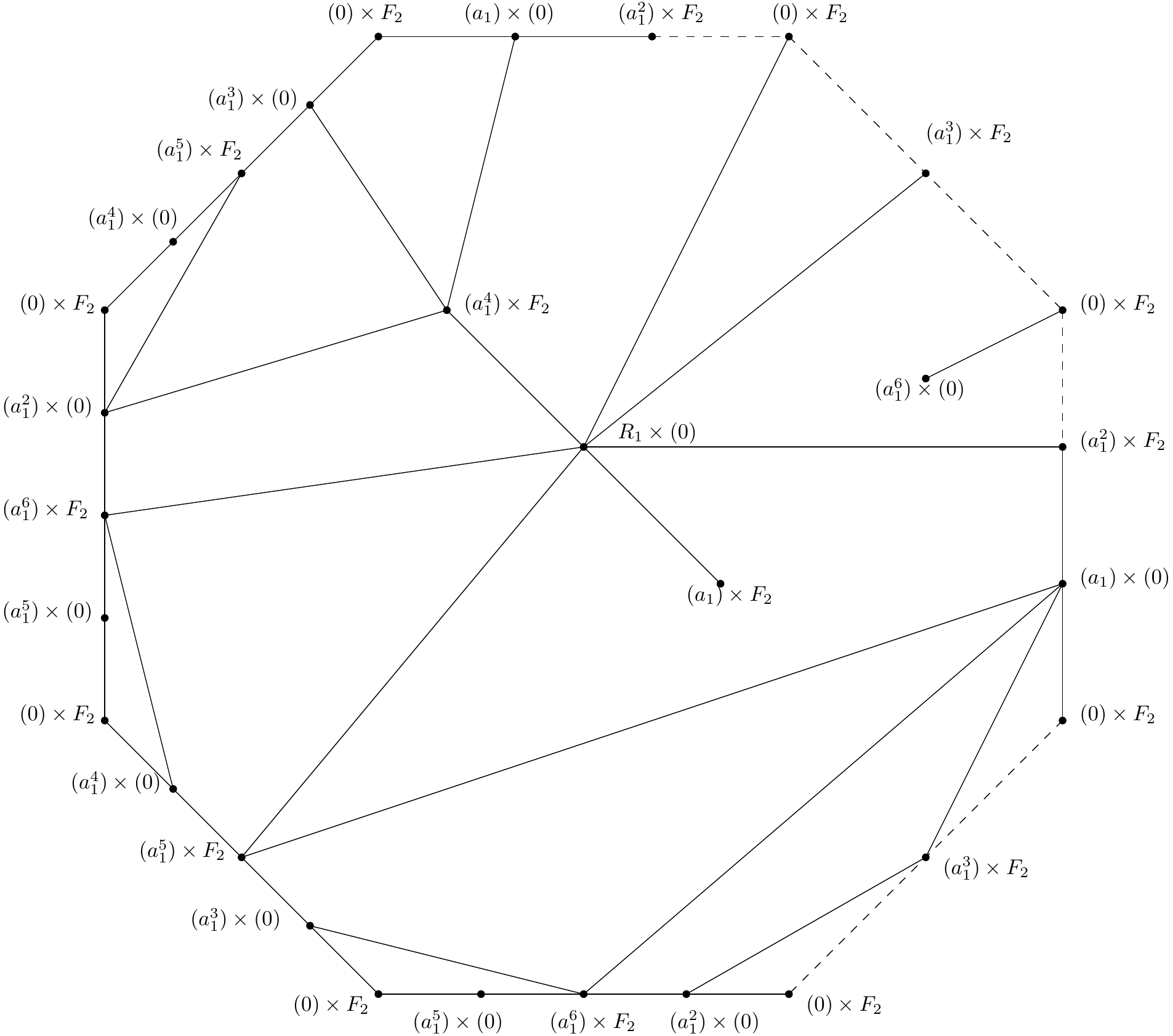}
\caption{Embedding of $\Gamma_r(R_1 \times F_2)$ in $\mathbb{S}_2$, where $\eta_1 = 7$.}
\label{genus2nilpotency7}
\end{figure}\begin{figure}[h!]
\centering
\includegraphics[width=0.6 \textwidth]{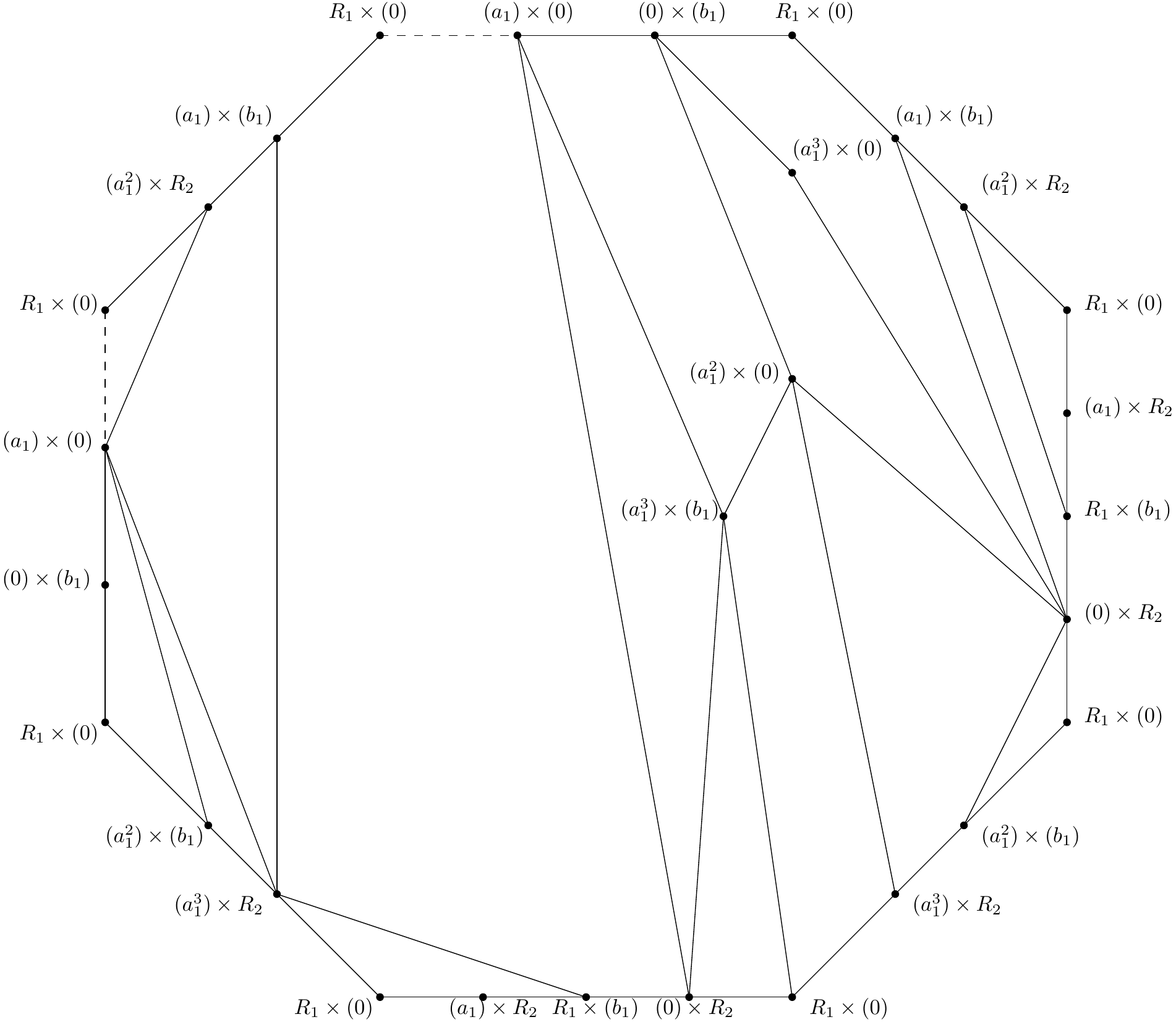}
\caption{Embedding of $\Gamma_r(R_1 \times R_2)$ in $\mathbb{S}_2$, where $\eta_1 = 4$ and  $\eta_2 = 2$.}
\label{genus2nilpotency4,2}
\end{figure}

\vspace{.3cm}
Now we classify all the non-local commutative rings $R$ such that the cozero-divisor graph $\Gamma'(R)$ is of genus $2$. The following lemma is essential for the proof of the Theorem \ref{genusofR1R2}.

\begin{lemma}\label{genusofR1R2R3}
   Let $R \cong R_1 \times R_2 \times \cdots \times R_n$ $(n \geq 3)$ be a non-local commutative ring. Then $g(\Gamma'(R)) \neq 2$. 
\end{lemma}

\begin{proof}
    On contrary, suppose that $g(\Gamma'(R)) = 2$. First assume that $n\ge 4$. Consider the vertices $x_1= (0,1,0,1,0, \ldots,0)$, $x_2= (0,1,0,0, \ldots 0)$, $x_3= (0,1,1,1,0,\ldots,0)$, $x_4= (0,1,1,0, \ldots,0)$, $x_5= (0,0,1,1,0,\ldots,0)$, $x_6= (1,0,0,1,0, \ldots,0)$, $x_7= (1,0,0,0, \ldots,0)$, $x_8= (1,0,1,1,0, \ldots,0)$, $x_9= (1,0,1,0, \ldots,0)$, $x_{10}= (1,1,0,0, \ldots,0)$, $x_{11}= (0,0,1,0, \ldots,0)$ and $x_{12}= (1,1,0,1,0, \ldots,0)$ of $\Gamma'(R)$. Then notice that the graph induced by the set $A = \{x_1$,$x_2$,$x_3$,$x_4$,$x_5$,$x_6$,$x_7$,$x_8$, $x_9$,$x_{10}$,$x_{11}$,$x_{12}\}$ contains a subgraph homeomorphic to $K_{5,5}$ (see Figure \ref{genusK55}). By Proposition \ref{genus}, we have $g(\Gamma'(R)) \ge 3$, a contradiction. It follows that $n=3$ and so $R \cong R_1 \times R_2 \times R_3$. Assume that $|R_i| \ge 3$ for each $i \in \{ 1,2,3\}$. Consider the set $B =  \{y_1, y_2, y_3, y_4, y_5, z_1, z_2, z_3, z_4, z_5 \}$, where $y_1 = (0,1,0)$, $y_2 = (0,1,1)$, $y_3 = (0,1,w)$, $y_4 = (0,v,0)$, $y_5 = (0,v,1)$, $z_1= (1,0,0)$, $z_2 = (1,0,1)$, $z_3 = (1,0,w)$, $z_4 = (u,0,0)$ and $z_5 = (u,0,1)$ such that $u \in R_1 \setminus \{0, 1\}$, $v \in R_2 \setminus \{0, 1\}$ and $ w \in R_3 \setminus \{0, 1\}$. Then the graph induced by the set $B$ has a subgraph isomorphic to $K_{5,5}$, a contradiction. It follows that $|R_i| \le 2$ for some $i$. Without loss of generality assume that $|R_1| \ge 4$, $|R_2| \ge 3$ and $|R_3| =2$. Then the graph induced by the vertices $x_1 = (0,1,0)$, $x_2 = (1,1,0)$, $x_3 = (u_1,1,0)$, $x_4 = (u_2,1,0)$, $x_5 = (0,v,0)$, $x_6= (1,v,0)$, $x_7 = (u_1,v,0)$, $x_8 = (u_2,v,0)$, $y_1 = (0,0,1)$, $y_2 = (1,0,1)$, $y_3 = (u_1,0,1)$ and $y_4 = (u_2,0,1)$, where $u_1, u_2 \in R_1 \setminus \{0, 1\}$, $ v \in R_2 \setminus \{0, 1\}$, is isomorphic to $K_{8,4}$, a contradiction. Thus, either $|R_1| < 4$ or $|R_2| <3$. Next, assume that $|R_2| = 2$. If $|R_1| \ge 5$, then consider the set $D =  \{y_1, y_2, y_3, y_4, y_5, z_1, z_2, z_3, z_4, z_5 \}$, where $y_1 = (0,1,0)$, $y_2 = (u_1,1,0)$, $y_3 = (u_2,1,0)$, $y_4 = (u_3,1,0)$, $y_5 = (1,1,0)$, $z_1= (0,0,1)$, $z_2 = (u_1,0,1)$, $z_3 = (u_2,0,1)$, $z_4 = (u_3,0,1)$ and $z_5 = (1,0,1)$ such that $u_i \in R_1 \setminus \{0, 1\}$ for each $i$ ($1 \le i \le 3$). The graph induced by the set $D$ contains a subgraph isomorphic to $K_{5,5}$, which is not possible. By Theorem \ref{cozeroplanarall} and Theorem \ref{cozerogenusfields}, $R$ is isomorphic to $\mathbb{Z}_3 \times \mathbb{Z}_3 \times \mathbb{Z}_2$ or $R_1 \times \mathbb{Z}_2 \times \mathbb{Z}_2$, where $|R_1|=4$. First, let $R \cong \mathbb{Z}_3 \times \mathbb{Z}_3 \times \mathbb{Z}_2$. Consider the set $X = \{x_1, x_2, x_3, x_4, x_5, y_1, y_2, y_3, y_4, y_5, z_1, z_2 \}$, where $x_1 = (1,0,0) $, $x_2 = (1,1,0) $, $x_3 = (1,2,0) $, $x_4 = (2,1,0) $, $x_5 = (2,2,0) $, $y_1 = (1,0,1) $, $y_2 = (2,0,1) $, $y_3 = (0,1,1) $, $y_4 = (0,2,1) $, $y_5 = (0,0,1) $, $z_1 = (0,1,0) $ and $z_2 = (0,2,0)$. Then the subgraph induced by the set $X$ is homeomorphic to $K_{5,5}$, where $\{ x_1, x_2, x_3, x_4, x_5 \}$ and $\{ y_1, y_2, y_3, y_4, y_5\}$ are the two partition sets such that $x_1 \sim z_1 \sim y_1$ and $x_1 \sim z_2 \sim y_2$. Consequently, $g(\Gamma'(\mathbb{Z}_3 \times \mathbb{Z}_3 \times \mathbb{Z}_2)) \ge 3$, a contradiction.

 Now, let $R \cong R_1 \times \mathbb{Z}_2 \times \mathbb{Z}_2$, where $|R_1|=4$. Consider the set $X =\{(0,1,1)$, $(u_1,0,1)$, $(u_2,0,1)$, $(u_1,1,0)$, $(u_2,1,0)$, $(1,1,0)\}$, where $u_1, u_2 \in R_1\setminus \{ 1, 0\}$. Let $H = \Gamma'(X)- \{e_1, e_2\}$, where $e_1 = \big((0,1,1),(u_1,0,1)\big)$ and $e_2 = \big((0,1,1),(u_2,0,1)\big)$. Then $H$ is a subgraph of $\Gamma'(R)$ which is isomorphic to $K_{3,3}$ (see Figure \ref{K33subgraphcozero}).
    \begin{figure}[h!]
    \centering
    \includegraphics[width=0.5 \textwidth]{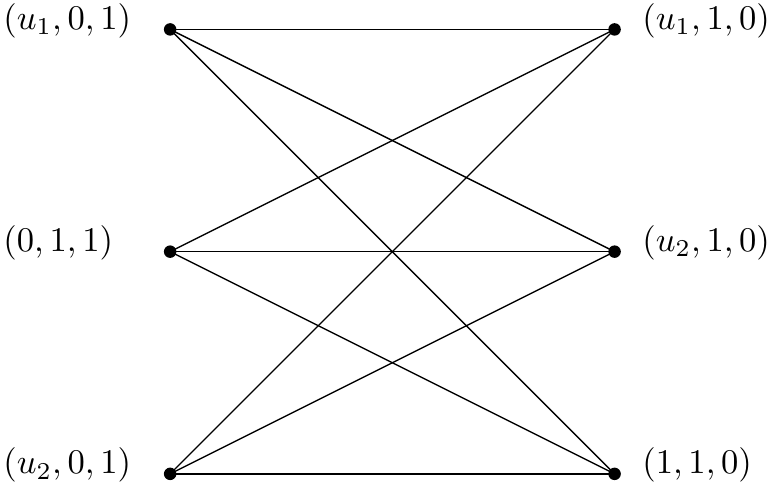}
    \caption{ The subgraph $H$ of $\Gamma'(R_1 \times \mathbb{Z}_2 \times \mathbb{Z}_2)$, where $|R_1| =4$.}
    \label{K33subgraphcozero}
    \end{figure}
It follows that $g(\Gamma'(H)) = 1$ and so $g(\Gamma'(R)) \ge 1$. Suppose that $g(\Gamma'(R)) = 1$. Now to embed $\Gamma'(R)$ in $\mathbb{S}_1$, first we insert the vertices $(1,0,1)$, $(0,1,0)$, $(0,0,1)$ and their incident edges in an embedding of $H$ in $\mathbb{S}_1$. Note that $(1,0,1) \sim (0,1,0) \sim (0,0,1)$. These vertices must be inserted in the same face $F'$. Both the vertices $(1,0,1)$, and $(0,0,1)$ are adjacent with the vertices $(u_1,1,0)$, $(u_2,1,0)$ and $(1,1,0)$. It follows that $F'$ must contain the vertices $(u_1,1,0)$, $(u_2,1,0)$ and $(1,1,0)$. Consequently, the insertion of the vertices $(1,0,1)$ and $(0,0,1)$ and their incident edges in $F'$ is not possible without edge crossings. Let $H'$ be the graph induced by the vertex set $V(H) \cup \{ (1,0,1),  (0,1,0), (0,0,1)\} $. It follows that $g(\Gamma'(H')) \ge 2$ and so $g(\Gamma'(R)) \ge 2$. Suppose that $g(\Gamma'(H')) = g(\Gamma'(R)) = 2$. To embed $\Gamma'(R)$ in $\mathbb{S}_2$, we insert the vertices $(u_1,0,0)$, $(u_2,0,0)$, $(1,0,0)$ and their incident edges in an embedding of $H'$ in $\mathbb{S}_2$. Note that all three vertices $(u_1,0,0)$, $(u_2,0,0)$ and $(1,0,0)$ are adjacent with the vertices $(0,1,0)$, $(0,1,1)$ and $(0,0,1)$. Since the degree of each of the vertex  $(0,1,0)$, $(0,1,1)$ and $(0,0,1)$ is greater than $2$ in $H'$, therefore these vertices can be common in exactly one face $F''$. It follows that the vertices $(u_1,0,0)$, $(u_2,0,0)$ and $(1,0,0)$ must be inserted in $F''$, which is not possible without edge crossings. Therefore, $g(\Gamma'(R)) \ge 3$, a contradiction. This completes our proof.
\end{proof}

 \noindent\textbf{Proof of Theorem \ref{genusofR1R2}:}
  Let $R \cong R_1 \times R_2 \times \cdots \times R_n$ $(n \ge 2)$ be a non-local commutative ring, where each $R_i$ is a local ring with maximal ideal $\mathcal{M}_i$. First suppose that $g(\Gamma'(R)) =2$. By Lemma \ref{genusofR1R2R3}, we have $R \cong R_1 \times R_2$. If both $R_1$ and $R_2$ are fields, then note that $\Gamma'(R) \cong K_{|R_1|-1, |R_2|-1}$. Consequently, by Proposition \ref{genus}, $R$ is isomorphic to one of the rings: $\mathbb{F}_4 \times \mathbb{F}_8$, $ \mathbb{F}_4 \times \mathbb{F}_9$, $ \mathbb{F}_4 \times \mathbb{Z}_{11}$, $ \mathbb{Z}_5 \times \mathbb{Z}_7$.

Now suppose that both $R_1$ and $R_2$ are not fields. Let $|R_1| \ge 4$ and $|R_2| \ge 8$. Then $|U(R_1)| \ge 2$ and $|U(R_2)| \ge 4$. Consider the set $X = \{ x_1, x_2, x_3, x_4, x_5, x_6, x_7, x_8, y_1, y_2, y_3, y_4 \}$, where $x_1= (u_1,0)$, $x_2= (u_2,0)$, $x_3= (a,0)$, $x_4= (a,b_1)$, $x_5= (a,b_2)$, $x_6= (u_1,b_1)$, $x_7= (u_1,b_2)$, $x_8= (u_2,b_1)$, $y_1= (0,v_1)$, $y_2= (0,v_2)$, $y_3= (0,v_3)$ and $y_4= (0,v_4)$ such that $a \in  \mathcal{M}_1^*$, $b_1, b_2 \in  \mathcal{M}_2^*$, $u_i \in U(R_1)$, $v_j \in U(R_2)$ for each $i,j$. Then the graph induced by the set $X$ contains a subgraph isomorphic to $K_{8,4}$. By Proposition \ref{genus}, we have $g(\Gamma'(R)) \ge 3$, a contradiction. Consequently, $R$ is isomorphic to one of the rings: $\mathbb{Z}_4 \times \mathbb{Z}_4$, $ \mathbb{Z}_4 \times \dfrac{\mathbb{Z}_2[x]}{( x^2 )}$, $ \dfrac{\mathbb{Z}_2[x]}{( x^2 )} \times \dfrac{\mathbb{Z}_2[x]}{( x^2 )}$.

Without loss of generality, suppose that $R_2$ is a field and $R_1$ is not a field i.e. $R \cong R_1 \times F_2$. Assume that $|R_1| \ge 8$ and $|F_2| \ge 3$. If $|\mathcal{M}_1^*| = 2$, then $|U(R_1)| \ge 5$. Consider the set $A = \{ x_1, x_2, x_3, x_4, x_5, y_1, y_2, y_3, y_4, y_5 \}$, where $x_1= (u_1,0)$, $x_2= (u_2,0)$, $x_3= (u_3,0)$, $x_4= (u_4,0)$, $x_5= (u_5,0)$, $y_1= (0,v_1)$, $y_2= (0,v_2)$, $y_3= (a,v_1)$, $y_4= (a,v_2)$ and $y_5= (b,v_1)$ such that $a,b \in  \mathcal{M}_1^*$, $u_i \in U(R_1)$, $v_j \in U(R_2)$ for each $i,j$. Then the graph induced by the set $A$ contains a subgraph isomorphic to $K_{5,5}$, a contradiction. If $|\mathcal{M}_1^*| \ge 3$, then $|U(R_1)| \ge 4$.  Consider the vertices $x_1= (u_1,0)$, $x_2= (u_2,0)$, $x_3= (u_3,0)$, $x_4= (u_4,0)$, $y_1= (0,v_1)$, $y_2= (0,v_2)$, $y_3= (a,v_1)$, $y_4= (a,v_2)$, $y_5= (b,v_1)$, $y_6= (b,v_2)$, $y_7= (c,v_1)$ and $y_8= (c,v_2)$, where $a,b,c \in  \mathcal{M}_1^*$, $u_i \in U(R_1)$, $v_j \in U(R_2)$ for each $i,j$. Note that, the graph induced by the set $\{ x_1, x_2, x_3, x_4, y_1, y_2, y_3, y_4, y_5, y_6, y_7, y_8 \}$ contains a subgraph isomorphic to $K_{4,8}$, which is not possible. Thus, either $|R_1| < 8$ or $|F_2| =2$.

First suppose that $|R_1| < 8$. Since $R_1$ is a local ring which is not a field, we have $|R_1| = 4$. If $2\le |F_2| \le 7$, then by Theorem \ref{cozeroplanarall} and Theorem \ref{cozerogenuslocalfield}, we have $g(\Gamma'(R)) \le 1$. Assume that $|F_2| \ge 13 $. Consider the set $Z = \{ x_1$,$x_2$,$x_3$,$y_1$,$y_2$,$y_3$,$y_4$,$y_5$,$y_6$,$y_7$,$y_8$,$y_9$,$y_{10}$,$y_{11}$,$y_{12} \}$, where $x_1= (u_1, 0)$, $x_2= (u_2, 0)$, $x_3= (a,0)$, $y_1= (0,v_1)$, $y_2= (0,v_2)$, $y_3= (0,v_3)$, $y_4= (0,v_4)$, $y_5= (0,v_5)$, $y_6= (0,v_6)$, $y_7= (0,v_7)$, $y_8= (0,v_8)$, $y_9= (0,v_9)$, $y_{10}= (0,v_{10})$, $y_{11}= (0,v_{11})$ and $y_{12}= (0,v_{12})$ such that $a \in  \mathcal{M}_1^*$, $u_i \in U(R_1)$, $v_j \in U(R_2)$ for each $i,j$. Then the subgraph induced by the set $Z$ is isomorphic to $K_{3,12}$, a contradiction. Consequently, $|F_2|= 8,9, 11$. Therefore, $R$ is isomorphic to one of the rings: 
\[ \mathbb{Z}_4 \times \mathbb{F}_8, \ \dfrac{\mathbb{Z}_2[x]}{( x^2 )}\times \mathbb{F}_8, \ \mathbb{Z}_4 \times \mathbb{F}_9, \  \dfrac{\mathbb{Z}_2[x]}{( x^2 )} \times \mathbb{F}_9, \ \mathbb{Z}_4 \times \mathbb{Z}_{11}, \ \dfrac{\mathbb{Z}_2[x]}{( x^2 )}\times \mathbb{Z}_{11}. \]

We may now suppose that $|F_2| = 2$. If $|R_1| = 8,9$, Theorem \ref{cozerogenuslocalfield}, $g(\Gamma'(R)) = 1$. Let $|R_1| \ge 16$. Assume that $|\mathcal{M}_1^*| \ge 4$. It implies that $|U(R_1)| \ge 5$. Consider the set $X = \{ x_1,x_2, x_3, x_4, x_5, y_1, y_2, y_3, y_4, y_5 \}$, where $x_1= (u_1,0)$, $x_2= (u_2,0)$, $x_3= (u_3,0)$, $x_4= (u_4,0)$, $x_5= (u_5,0)$, $y_1= (0,1)$, $y_2= (a_1,1)$, $y_3= (a_2,1)$, $y_4= (a_3,1)$ and $y_5= (a_4,1)$ such that $u_i \in U(R_1)$, $a_j \in  \mathcal{M}_1^*$ for each $i,j$. Then the graph induced by the set $X$ has a subgraph isomorphic to $K_{5,5}$, a contradiction. Now, let $|\mathcal{M}_1^*| = 3$. Then $|U(R_1)| \ge 12 $. Consider the vertices $x_1= (u_1,0)$, $x_2= (u_2,0)$, $x_3= (u_3,0)$, $x_4= (u_4,0)$, $x_5= (u_5,0)$, $x_6= (u_6,0)$, $x_7= (u_7,0)$, $x_8= (u_8,0)$, $y_1= (0,1)$, $y_2= (a_1, 1)$, $y_3= (a_2,1)$ and $y_4= (a_3,1)$, where $u_i \in U(R_1)$, $a_j \in  \mathcal{M}_1^*$ for each $i,j$. Then the graph induced by the set $\{ x_1, x_2, x_3, x_4, x_5, x_6, x_7, x_8, y_1, y_2, y_3, y_4\}$ contains a subgraph isomorphic to $K_{8,4}$, which is not possible. Thus $|F_2| \neq 2$. 

Converse follows from the Figures \ref{genusfig1}-\ref{genusfig4}, in which $a \in \mathcal{M}_1^*$, $b \in \mathcal{M}_2^*$, $u_i \in U(R_1)$ and $v_j \in U(R_2)$ for each $i,j$.  

\begin{figure}[h!]
\centering
\includegraphics[width=0.57 \textwidth]{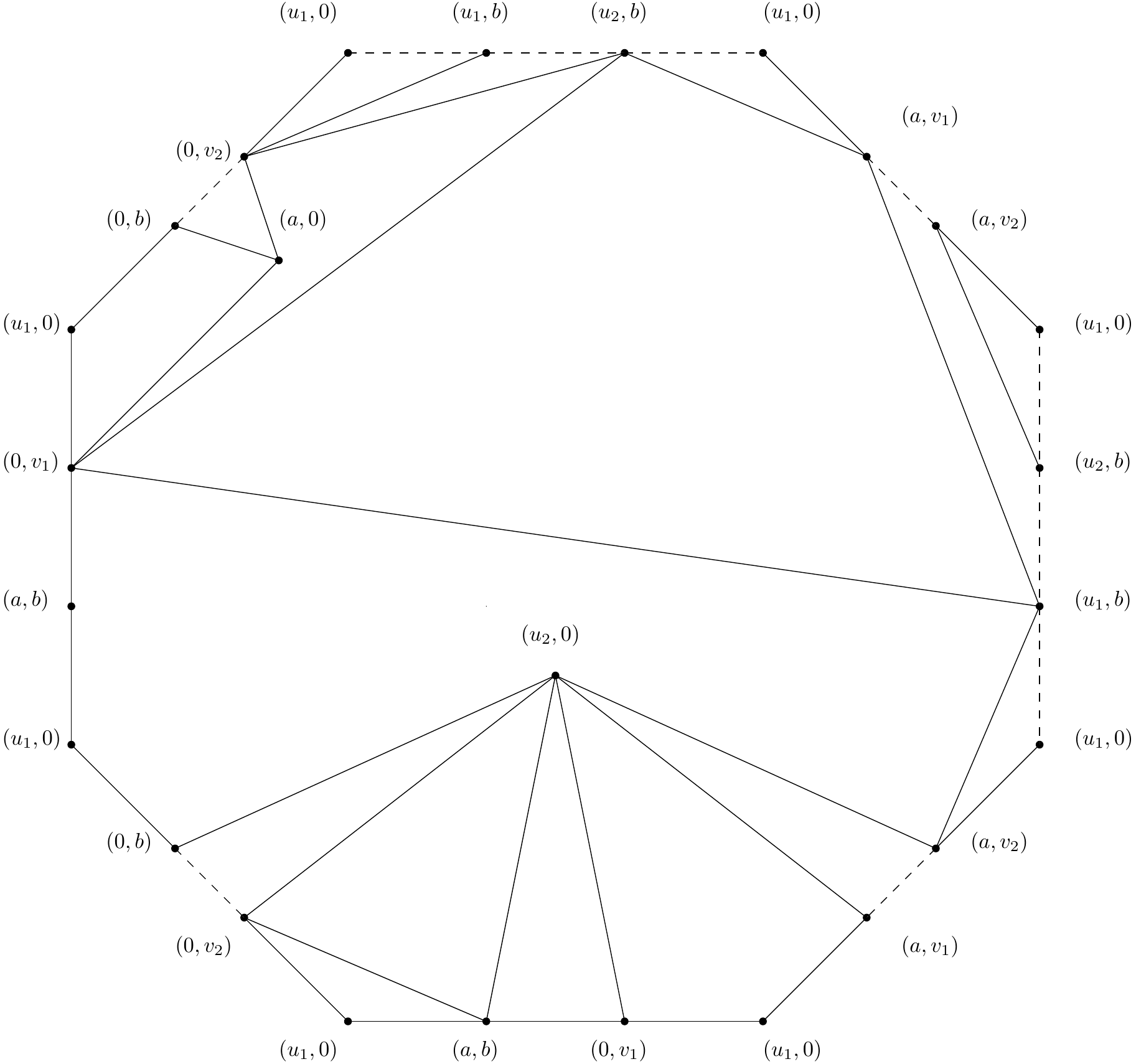}
\caption{Embedding of $\Gamma'(\mathbb{Z}_4 \times \mathbb{Z}_4)$, $\Gamma'\left(\mathbb{Z}_4 \times \dfrac{\mathbb{Z}_2[x]}{( x^2 )}\right)$ and $\Gamma'\left(\dfrac{\mathbb{Z}_2[x]}{( x^2 )} \times \dfrac{\mathbb{Z}_2[x]}{( x^2 )}\right)$ in $\mathbb{S}_2$.}
\label{genusfig1}
\end{figure}
\begin{figure}[h!]
\centering
\includegraphics[width=0.57 \textwidth]{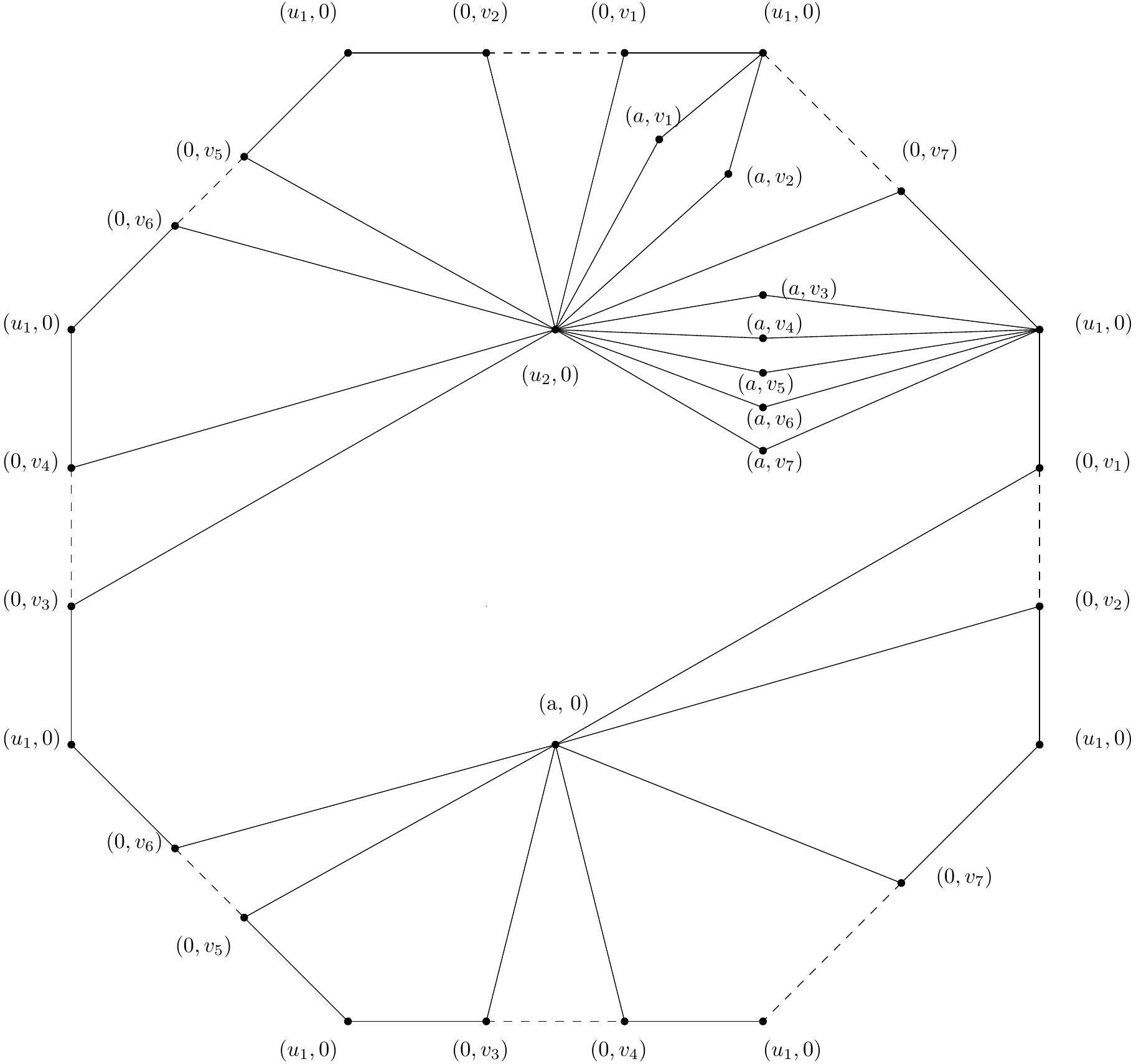}
\caption{Embedding of $\Gamma'(\mathbb{Z}_4 \times \mathbb{F}_8)$ and $\Gamma'\left(\dfrac{\mathbb{Z}_2[x]}{( x^2 )} \times \mathbb{F}_8\right)$ in $\mathbb{S}_2$.}
\label{genusfig2}
\end{figure}
\begin{figure}[h!]
\centering
\includegraphics[width=0.57 \textwidth]{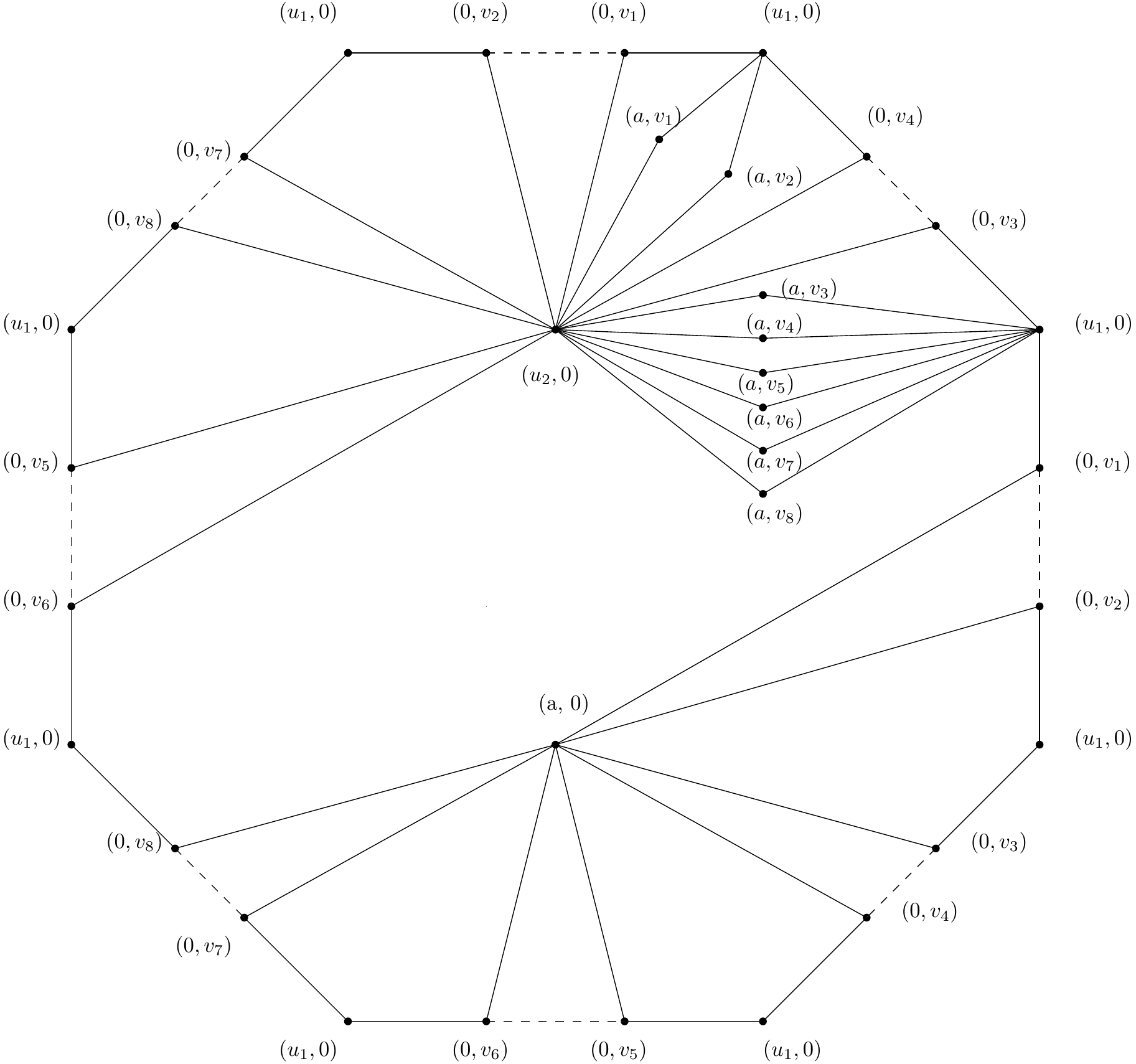}
\caption{Embedding of $\Gamma'(\mathbb{Z}_4 \times \mathbb{F}_9)$ and $\Gamma'\left(\dfrac{\mathbb{Z}_2[x]}{( x^2 )} \times \mathbb{F}_9\right)$ in $\mathbb{S}_2$.}
\label{genusfig3}
\end{figure}
\begin{figure}[h!]
\centering
\includegraphics[width=0.57 \textwidth]{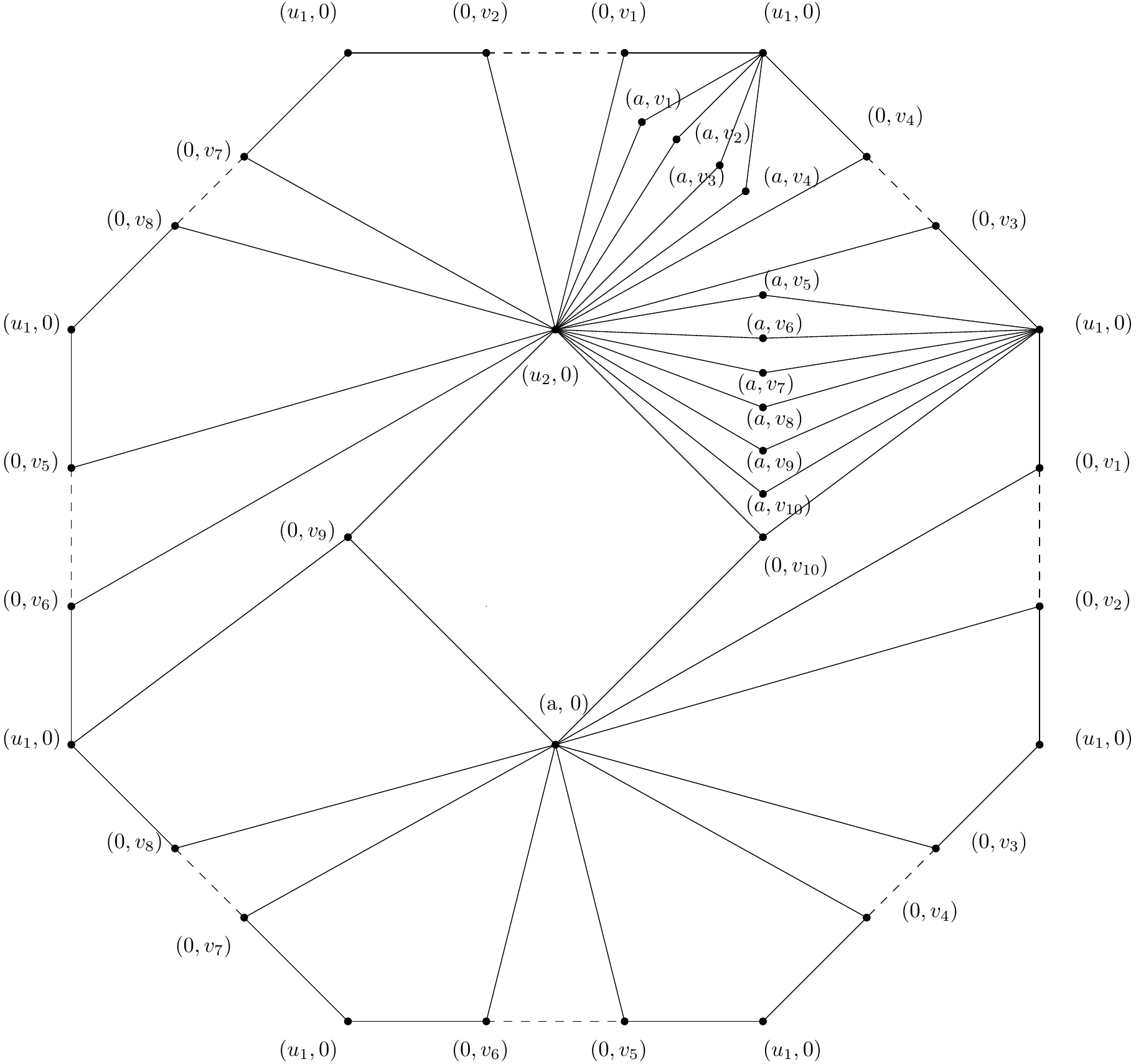}
\caption{Embedding of $\Gamma'(\mathbb{Z}_4 \times \mathbb{Z}_{11})$ and $\Gamma'\left(\dfrac{\mathbb{Z}_2[x]}{( x^2 )} \times \mathbb{Z}_{11}\right)$ in $\mathbb{S}_2$.}
\label{genusfig4}
\end{figure}

\newpage


\end{document}